\documentclass[11pt]{article}
\usepackage{amsthm,amsfonts, amsbsy, amssymb,amsmath,graphicx}
\usepackage{graphics}
\usepackage[english]{babel}
\usepackage{graphicx}
\usepackage{lineno}
\usepackage{enumerate}
\usepackage{tikz}
\usepackage{tkz-graph}
\usepackage{tkz-berge}
\usepackage{hyperref}
\usepackage{color}
\usepackage{tikz}

\hypersetup{colorlinks=true}

\hypersetup{colorlinks=true, linkcolor=blue, citecolor=blue,urlcolor=blue}

\title{Restrained Italian domination in graphs}
\date{}
\author {
Babak Samadi$^a$, Morteza Alishahi$^b$, Iman Masoumi$^c$ and Doost Ali Mojdeh$^d$\thanks{Corresponding author}\vspace{3mm}\\
$^a$Babolsar, Iran\\
{\tt samadibabak62@gmail.com}\vspace{2mm}\\
$^b$Department of Mathematics\\
Islamic Azad University, Nazarabad Branch, Nazarabad, Iran\\
{\tt morteza.alishahi@gmail.com}\vspace{2mm}\\
$^c$Department of Mathematics\\
University of Tafresh, Tafresh, Iran\\
{\tt i$_{-}$masoumi@yahoo.com}\vspace{2mm}\\
$^d$Department of Mathematics\\
University of Mazandaran, Babolsar, Iran\\
{\tt damojdeh@umz.ac.ir}\vspace{2mm}\\
}
\date{}

\newtheorem{theorem}{Theorem}[section]

\newtheorem{lemma}[theorem]{Lemma}

\newtheorem{proposition}[theorem]{Proposition}

\theoremstyle{definition}

\theoremstyle{remark}

\begin{document}

\maketitle

\begin{abstract} For a graph $G=(V(G),E(G))$, an Italian dominating function (ID function) $f:V(G)\rightarrow\{0,1,2\}$ has the property that for every vertex $v\in V(G)$ with $f(v)=0$, either $v$ is adjacent to a vertex assigned $2$ under $f$ or $v$ is adjacent to least two vertices assigned $1$ under $f$. The weight of an ID function is $\sum_{v\in V(G)}f(v)$. The Italian domination number is the minimum weight taken over all ID functions of $G$.

In this paper, we initiate the study of a variant of ID functions. A restrained Italian dominating function (RID function) $f$ of $G$ is an ID function of $G$ for which the subgraph induced by $\{v\in V(G)\mid f(v)=0\}$ has no isolated vertices, and the restrained Italian domination number $\gamma_{rI}(G)$ is the minimum weight taken over all RID functions of $G$. We first prove that the problem of computing this parameter is NP-hard, even when restricted to bipartite graphs and chordal graphs as well as planar graphs with maximum degree five. We prove that $\gamma_{rI}(T)$ for a tree $T$ of order $n\geq3$ different from the double star $S_{2,2}$ can be bounded from below by $(n+3)/2$. Moreover, all extremal trees for this lower bound are characterized in this paper. We also give some sharp bounds on this parameter for general graphs and give the characterizations of graphs $G$ with small or large $\gamma_{rI}(G)$.

\noindent \ \ \
\end{abstract}
\vspace{1mm}


\section{Introduction and preliminaries}

Throughout this paper, we consider $G$ as a finite simple graph with vertex set $V(G)$ and edge set $E(G)$. We use \cite{w} as a reference for terminology and notation which are not explicitly defined here. The {\em open neighborhood} of a vertex $v$ is denoted by $N(v)$, and its {\em closed neighborhood} is $N[v]=N(v)\cup \{v\}$. The {\em minimum} and {\em maximum degrees} of $G$ are denoted by $\delta(G)$ and $\Delta(G)$, respectively. Given subsets $A,B\subseteq V(G)$, by $[A,B]$ we mean the set of all edges with one end point in $A$ and the other in $B$. For a given subset $S\subseteq V(G)$, by $G[S]$ we represent the subgraph induced by $S$ in $G$. A tree $T$ is a {\em double star} if it contains exactly two vertices that are not leaves. A double star with $p$ and $q$ leaves attached to each support vertex, respectively, is denoted by $S_{p,q}$.

A set $S\subseteq V(G)$ is called a {\em dominating set} if every vertex not in $S$ has a neighbor in $S$. The {\em domination number} $\gamma(G)$ of $G$ is the minimum cardinality among all dominating sets of $G$. A {\em restrained dominating set} (RD set) in a graph $G$ is a dominating set $S$ in $G$ for which every vertex in $V(G)\setminus S$ is adjacent to another vertex in $V(G)\setminus S$. The {\em restrained domination number} (RD number) of $G$, denoted by $\gamma_{r}(G)$, is the smallest cardinality of an RD set of $G$. This concept was formally introduced in \cite{dhhlr} (albeit, it was indirectly introduced in \cite{tp}).

For a function $f:V(G)\rightarrow\{0,1,2\}$, we let $V^{f}_{i}=\{v\in V(G)\mid f(v)=i\}$ for each $0\leq i\leq2$ (we simply write $f=(V_{0},V_{1},V_{2})$ if there is no ambiguity with respect to the function $f$ and the graph $G$). We call $\omega(f)=f(V(G))=\sum_{v\in V(G)}f(v)$ as the {\em weight} of $f$. A {\em Roman dominating function} (RD function) of a graph $G$ is a function $f:V(G)\rightarrow\{0,1,2\}$ such that if $f(v)=0$ for some $v\in V(G)$, then there exists $w\in N(v)$ such that $f(u)=2$ (\cite{cdhh}). In 2015, Pushpam and Padmapriea (\cite{pp}) introduced the concept of restrained Roman domination in graphs as follows. An RD function $f:V(G)\rightarrow\{0,1,2\}$ is called a {\em restrained Roman dominating function} (RRD function for short) if $G[V_{0}]$ has no isolated vertices. The {\em restrained Roman domination number} (RRD number) $\gamma_{rR}(G)$ is the minimum weight $\sum_{v\in V(G)}f(v)$ of an RRD function $f$ of $G$.

Chellali et al. in \cite{chhm} introduced an Italian dominating function (also known as a Roman $\{2\}$-dominating function) $f$ as follows. An \textit{Italian dominating function} (ID function) is a function $f:V(G)\rightarrow \{0,1,2\}$ with the property that for every vertex $v\in V(G)$ with $f(v)=0$, it follows that $f(N(v))\geq 2$. That is, either there is a vertex $u\in N(v)$ with $f(u)=2$ or at least two vertices $x,y\in N(v)$ with $f(x)=f(y)=1$. A \textit{restrained Italian dominating function} (RID function) of $G$ is an ID function of $G$ for which $G[V_{0}]$ has no isolated vertices. The minimum weight of an RID function of $G$ is called the \textit{restrained Italian domination number} (RID number) of $G$, denoted by $\gamma_{rI}(G)$.

In this paper, we investigate the restrained Italian domination in graphs. We prove that the problem of computing the RID number is NP-hard even when restricted to some well-known families of graphs and give some sharp lower and upper bounds on this parameter. In section $3$, we prove that $\gamma_{rI}(T)\geq(n+3)/2$ for any tree $T\neq S_{2,2}$ of order $n\geq3$. Moreover, the characterization of all trees for which the equality holds is given in this paper. We also give the characterizations of graphs with small or large RID numbers.

By a $\gamma(G)$-set or a $\gamma_{r}(G)$-set, we mean a dominating set or a restrained dominating set in $G$ of cardinality $\gamma(G)$ or $\gamma_{r}(G)$, respectively. Also, a $\gamma_{rI}(G)$-function is an RID function $f$ of $G$ with weight $\omega(f)=\gamma_{rI}(G)$.

\section{Complexity and computational issues}

We consider the problem of deciding whether a graph $G$ has an RID function of weight at most a given integer. That is stated in the following decision problem.

$$\begin{tabular}{|l|}
  \hline
  \mbox{RISTRAINED ITALIAN DOMINATION problem (RID problem)}\\
  \mbox{INSTANCE: A graph $G$ and an integer $j\leq|V(G)|$.}\\
  \mbox{QUESTION: Is there an RID function $f$ of weight at most $j$?}\\
  \hline
\end{tabular}$$

In what follows, we make use of the DOMINATING SET problem which is known to be NP-complete for planar graphs with maximum degree three (\cite{gj}), bipartite graphs and chordal graphs (\cite{hhs}).

$$\begin{tabular}{|l|}
  \hline
  \mbox{DOMINATING SET problem}\\
  \mbox{INSTANCE: A graph $G$ and an integer $k\leq|V(G)|$.}\\
  \mbox{QUESTION: Is there a dominating set of cardinality at most $k$?}\\
  \hline
\end{tabular}$$\vspace{0.25mm}

\begin{theorem}\label{Comp}
The RID problem is NP-complete even when restricted to bipartite graphs, chordal graphs and planar graphs with maximum degree five.
\end{theorem}
\begin{proof}
The problem clearly belongs to NP since checking that a given function is indeed an RID function of weight at most $j$ can be done in polynomial time. Set $j=5n+k$. Let $G$ be a graph with $V(G)=\{v_{1},\cdots,v_{n}\}$. For any $1\leq i\leq n$, we add a new vertex $g_{i}$ and a double star $T_{i}$ with $V(T_{i})=\{a_{i},b_{i},c_{i},d_{i},e_{i},f_{i}\}$ in which $a_{i}$ and $b_{i}$ are the support vertices, $N_{T_{i}}(a_{i})\setminus\{b_{i}\}=\{c_{i},d_{i}\}$ and $N_{T_{i}}(b_{i})\setminus\{a_{i}\}=\{e_{i},f_{i}\}$. We then join $v_{i}$ to both $a_{i}$ and $g_{i}$, for all $1\leq i\leq n$. Let $G'$ be the constructed graph.

Let $f$ be a $\gamma_{rI}(G')$-function. Clearly, $f(V(T_{i})\cup\{g_{i}\})\geq5$ for all $1\leq i\leq n$. Moreover, if there exists a vertex $v_{j}\in V(G)\cap V_{0}$ which does not have any neighbor in $(V_{1}\cup V_{2})\cap V(G)$, then it is not difficult to see that $f(V(T_{j})\cup\{g_{j}\})\geq6$. We define $X$ to be the set of such vertices, that is, $X=\{v_{j}\in V(G)\cap V_{0}\mid  \mbox{$v$ has no neighbor in}\ (V_{1}\cup V_{2})\cap V(G)\}$. We have
\begin{equation}\label{INE1}
\begin{array}{lcl}
\gamma_{rI}(G')=\omega(f)&=&\sum_{i=1}^{n}f(V(T_{i})\cup\{g_{i}\})+f(V(G))\\
&=&\sum_{v_{i}\in V(G)\setminus X}f(V(T_{i})\cup\{g_{i}\})+\sum_{v_{i}\in X}f(V(T_{i})\cup\{g_{i}\})+f(V(G))\\
&\geq&5|V(G)\setminus X|+6|X|+|(V_{1}\cup V_{2})\cap V(G)|\\
&=&5n+|X|+|(V_{1}\cup V_{2})\cap V(G)|.
\end{array}
\end{equation}
On the other hand, $S=X\cup((V_{1}\cup V_{2})\cap V(G))$ is a dominating set in $G$. Therefore, $\gamma(G)\leq|S|=|X|+|(V_{1}\cup V_{2})\cap V(G)|$. By using the inequality (\ref{INE1}), we deduce that $\gamma_{rI}(G')\geq5n+\gamma(G)$.

Conversely, let $S'$ be a $\gamma(G)$-set. We define $f'$ by $f'(a_{i})=f'(b_{i})=f'(v)=0$ for each $v\in V(G)\setminus S'$, and $f'(x)=1$ for the other vertices $x$. It is readily checked that $f'$ is an RID function of $G'$ with weight $5n+|S'|$. Therefore, $\gamma_{rI}(G')\leq5n+\gamma(G)$. This shows that $\gamma_{rI}(G')=5n+\gamma(G)$.

Our reduction is now completed by taking into account the fact that $\gamma_{rI}(G')\leq j$ if and only if $\gamma(G)\leq k$. Since the DOMINATING SET problem is NP-complete for both bipartite graphs and chordal graphs, we have the same with the RID problem. Moreover, it is NP-complete for planar graphs with maximum degree five since the DOMINATING SET problem is NP-complete for planar graphs with maximum degree three.
\end{proof}

As a consequence of Theorem \ref{Comp}, we conclude that the problem of computing the RID number is NP-hard, even when restricted to bipartite graphs and chordal graphs as well as planar graphs with maximum degree five. In consequence, it would be desirable to bound the RID number in terms of several different invariants of graphs.

\begin{proposition}\label{Prop3}
For any connected graph $G$ of order $n\geq3$ and size $m$,
$$\gamma_{rI}(G)\geq min\{\gamma_{rR}(G),n-2m/5,n-(2m-5)/3\}.$$
\end{proposition}
\begin{proof}
Let $f=(V_{0},V_{1},V_{2})$ be a $\gamma_{rI}(G)$-function. If every vertex in $V_{0}$ is adjacent to a vertex in $V_{2}$, then $f$ is an RRD function of $G$. Therefore, $\gamma_{rI}(G)\geq \gamma_{rR}(G)$ (and so, $\gamma_{rI}(G)=\gamma_{rR}(G)$). So, we may assume that some vertices in $V_{0}$ do not have any neighbor in $V_{2}$. If $V_{2}=\emptyset$, then every vertex in $V_{0}$ is adjacent to at least two vertices in $V_{1}$. On the other hand, $|[V_{0},V_{0}]|\geq|V_{0}|/2$ since $G[V_{0}]$ has no isolated vertices. Therefore,
$$2m\geq2|[V_{0},V_{0}]|+2|[V_{0},V_{1}]|\geq5|V_{0}|.$$
We now have, $n=|V_{0}|+|V_{1}|\leq2m/5+\gamma_{rI}(G)$. Therefore, $\gamma_{rI}(G)\geq n-2m/5$. So, we assume that $V_{2}\neq\emptyset$ and $\gamma_{rI}(G)<\gamma_{rR}(G)$. In such a situation, at least one vertex in $V_{0}$ does not have any neighbor in $V_{2}$. We get
$$2m\geq2|[V_{0},V_{0}]|+2|[V_{0},V_{1}\cup V_{2}]|\geq|V_{0}|+2(|V_{0}|-1)+4.$$
Therefore, $|V_{0}|\leq2(m-1)/3$. We now have,
$$n=|V_{0}|+|V_{1}|+|V_{2}|\leq2(m-1)/3+\gamma_{rI}(G)-|V_{2}|\leq2(m-1)/3+\gamma_{rI}(G)-1$$
implying that $\gamma_{rI}(G)\geq n-(2m-5)/3$. This completes the proof.
\end{proof}

We conclude this section by showing that the lower bound given in Proposition \ref{Prop3} is sharp. For the sake of convenience, we let $\eta(G)=\mbox{min}\{\gamma_{rR}(G),n-2m/5,n-(2m-5)/3\}$. Let $G'$ be obtained from $k\geq2$ copies of $K_{2}$, a new vertex $v$ and joining $v$ to any vertex of the $k$ copies of $K_{2}$ (this graph was given in \cite{jk}). It is easy to see that $\gamma_{rI}(G')=2=\gamma_{rR}(G')=\eta(G')$. Moreover, $\gamma_{rI}(S_{2,2})=4=n-2m/5=\eta(S_{2,2})$. Also, if $G''=G'-x$ for any vertex $x\neq v$ and $k\geq4$, we deduce that $\gamma_{rI}(G'')=3=n-(2m-5)/3=\eta(G'')$.


\section{Trees}

Our main aim in this section is to bound the RID number of a tree from below just in terms of its order. Moreover, we characterize all trees attaining the bound.

\begin{theorem}\label{Tree}
Let $T$ be a tree of order $n\geq3$ different from the double star $S_{2,2}$. Then, $\gamma_{rI}(T)\geq(n+3)/2$.
\end{theorem}
\begin{proof}
We proceed by induction on the order $n\geq3$ of $T$. The result is obvious when $n=3$. Moreover, $\gamma_{rI}(K_{1,n-1})=n\geq(n+3)/2$ for $n\geq3$. Hence, we may assume that diam$(T)\geq3$. If diam$(T)=3$, then $T$ is isomorphic to a double star $S_{a,b}$ with $1\leq a\leq b$ where $(a,b)\neq(2,2)$. Then, it is easy to check that $\gamma_{rI}(S_{a,b})\geq(n+3)/2$. So, in what follows we may assume that diam$(T)\geq4$, which implies that $n\geq5$.

Suppose that $\gamma_{rI}(T')\geq(n'+3)/2$, for each tree $T'\neq S_{2,2}$ of order $3\leq n'<n$. Let $T\neq S_{2,2}$ be a tree of order $n$. Let $r$ and $v$ be two leaves of $T$ with $d(r,v)=\mbox{diam}(T)$. We root the tree $T$ at $r$. Let $u$ be the parent of $v$, and $w$ be the parent of $u$. Let $f=(V_{0},V_{1},V_{2})$ be a $\gamma_{rI}(T)$-function. From now on, we assume that $V(S_{2,2})=\{a,b,a_{1},a_{2},b_{1},b_{2}\}$ in which $a$ and $b$ are the support vertices, $a_{1}$ and $a_{2}$ are the leaves adjacent to $a$, and $b_{1}$ and $b_{2}$ are the leaves adjacent to $b$. For a vertex $x$ of $T$, by $T_{x}$ we mean the subtree of T rooted at $x$ consisting of $x$ and all its descendants in $T$. We now distinguish two cases depending on $f(u)$.\vspace{0.75mm}

\textit{Case 1.} $f(u)\geq1$. If $T'=T-v$ is isomorphic to the double star $S_{2,2}$, then $T$ is obtained from $S_{2,2}$ by joining a new vertex to a leaf of it. Consequently, $\gamma_{rI}(T)=5=(n+3)/2$. Therefore, we assume that $T'\neq S_{2,2}$. Moreover, $n(T')>3$ since diam$(T)\geq4$. On the other hand, $f(v)=1$ since $f(u)\geq1$. This shows that $f'=f\mid_{V(T')}$ is an RID function of $T'$. Using the induction hypothesis we have
\begin{equation}\label{EQU1}
\frac{n+2}{2}=\frac{n(T')+3}{2}\leq \gamma_{rI}(T')\leq \omega(f')=\gamma_{rI}(T)-1,
\end{equation}
which implies the lower bound.\vspace{0.75mm}

\textit{Case 2.} $f(u)=0$. Since $u$ is not an isolated vertex of $T[V_{0}]$, it follows that $f(w)=0$ and that $u$ is adjacent to at least one leaf $v'$ different form $v$ with $f(v')=1$ if $f(v)=1$, or $v$ is the only leaf adjacent to $u$ if $f(v)=2$. We now consider two other cases.\vspace{0.75mm}

\textit{Subcase 2.1.} Suppose that $N_{T}(w)\cap(V_{0}\setminus\{u\})=\emptyset$. Let $T''=T-V(T_{w})$. Suppose first that $T''=S_{2,2}$. Without loss of generality, we may assume that $w$ is adjacent to $b$ or $b_{2}$.\vspace{0.75mm}

\textit{Subcase 2.1.1.} $wb_{2}\in E(T)$. If deg$(w)=2$, then it is easy to see that $f(b_{2})=2$ and hence $\gamma_{rI}(T)\geq5+\ell_{u}\geq(n+3)/2$ in which $\ell_{u}$ is the number of leaves adjacent to $u$. So, let deg$(w)\geq3$. Notice that since $d(r,v)=$ diam$(T)$, all descendants of $w$ are leaves or support vertices. Moreover, all descendants of $w$ different from $u$ are assigned at least $1$ under $f$. Suppose that $p$ and $q$ are the number of children and grandchildren of $w$, respectively. We now have $n=n(S_{2,2})+n(T_{w})=p+q+7$. Furthermore, $\gamma_{rI}(T)\geq n-4$ if $\ell_{u}\geq2$, and $\gamma_{rI}(T)\geq n-3$ if $\ell_{u}=1$. In both cases, we have $\gamma_{rI}(T)\geq(n+3)/2$.\vspace{0.75mm}

\textit{Subcase 2.1.2.} $wb\in E(T)$. Suppose that deg$(w)=2$. In such a situation, we observe that $\gamma_{rI}(T)=n-2\geq(n+3)/2$ by assigning $0$ to the vertices $a$ and $b$, and $1$ to the other vertices. On the other hand, since $f(u)=f(w)=0$, it follows that $(f(a),f(b))=(1,2)$. Therefore, $\gamma_{rI}(T)=\omega(f)\geq n-1$ which is a contradiction. Thus, deg$(w)\geq3$. Note that the assignment $g(a)=g(b)=g(w)=0$ and $g(x)=1$ for any other vertex $x$ defines an RID function of $T$. So, $\gamma_{rI}(T)\leq n-3$. But the condition $N_{T}(w)\cap(V_{0}\setminus\{u\})=\emptyset$ implies that all descendants of $w$ different from $u$, as well as the vertices $a$ and $b$, are assigned at least $1$ under $f$. Therefore, $\gamma_{rI}(T)\geq n-2$. This is a contradiction. Therefore, $w$ is not adjacent to $b$.\vspace{0.75mm}

So, we assume that $T''\neq S_{2,2}$. On the other hand, diam$(T)\geq4$ implies that $n(T'')\geq2$. Let $n(T'')=2$. It is easy to observe that $\gamma_{rI}(T)\geq \ell_{u}+3\geq(n+3)/2$ when deg$(w)=2$. When deg$(w)\geq3$, we have $\gamma_{rI}(T)=n-1$ if $\ell_{u}=1$, and $\gamma_{rI}(T)=n-2$ if $\ell_{u}\geq2$. In both cases, we end up with $\gamma_{rI}(T)\geq(n+3)/2$. So, we may assume that $T''\neq S_{2,2}$ and $n(T'')\geq3$. Then, $n(T'')=n-p-q-1$. Furthermore, $f''=f\mid_{V(T'')}$ is an RID function of $T''$. We consider two cases depending on $\ell_{u}$.\vspace{0.75mm}

\textit{Subcase 2.1.3.} $\ell_{u}=1$. Then, $\omega(f'')=\gamma_{rI}(T)-p-q$. So, we get
\begin{equation}\label{EQU2}
\frac{n-p-q-1+3}{2}=\frac{n(T'')+3}{2}\leq \gamma_{rI}(T'')\leq \omega(f'')=\gamma_{rI}(T)-p-q.
\end{equation}
Therefore, $\gamma_{rI}(T)\geq(n+p+q+2)/2>(n+3)/2$.\vspace{0.75mm}

\textit{Subcase 2.1.4.} $\ell_{u}\geq2$. We have $\omega(f'')=\gamma_{rI}(T)-p-q+1$. Therefore,
\begin{equation}\label{EQU3}
\frac{n-p-q-1+3}{2}=\frac{n(T'')+3}{2}\leq \gamma_{rI}(T'')\leq \omega(f'')=\gamma_{rI}(T)-p-q+1.
\end{equation}
Consequently, $\gamma_{rI}(T)\geq(n+p+q)/2\geq(n+3)/2$.\vspace{0.75mm}

\textit{Subcase 2.2.} Suppose that $N_{T}(w)\cap(V_{0}\setminus\{u\})\neq\emptyset$. We set $T'''=T-V(T_{u})$. Since diam$(T)\geq4$, it follows that $n(T''')\geq3$. If $T'''=S_{2,2}$, then we may assume that $u$ is adjacent to $b_{2}$ or $b$. Assume that $ub_{2}\in E(T)$, that is, $w=b_{2}$. In such a situation, the condition $N_{T}(w)\cap(V_{0}\setminus\{u\})\neq\emptyset$ implies that $f(u)=f(v)=f(w)=0$, a contradiction. Thus, $ub\in E(T)$. We have $\gamma_{rI}(T)\geq \ell_{u}+4$ if $\ell_{u}\geq2$, and $\gamma_{rI}(T)=6$ if $\ell_{u}=1$. In both cases it results in $\gamma_{rI}(T)\geq(n+3)/2$. So, let $T'''\neq S_{2,2}$. We have $n(T''')=n-\ell_{u}-1$ and that $f'''=f\mid_{V(T''')}$ is an RID function with weight at most $\gamma_{rI}(T)-\ell_{u}$. Therefore,
\begin{equation*}
\frac{n-\ell_{u}-1+3}{2}=\frac{n(T'')+3}{2}\leq \gamma_{rI}(T''')\leq \omega(f''')\leq \gamma_{rI}(T)-\ell_{u}.
\end{equation*}
So, $\gamma_{rI}(T)\geq(n+\ell_{u}+2)\geq(n+3)/2$.

All in all, we have proved the desired lower bound.
\end{proof}

In what follows we characterize all extremal trees for the lower bound given in Theorem \ref{Tree}. For this purpose, we introduce the family $\mathcal{J}$ of trees depicted in Figure \ref{fig1}.

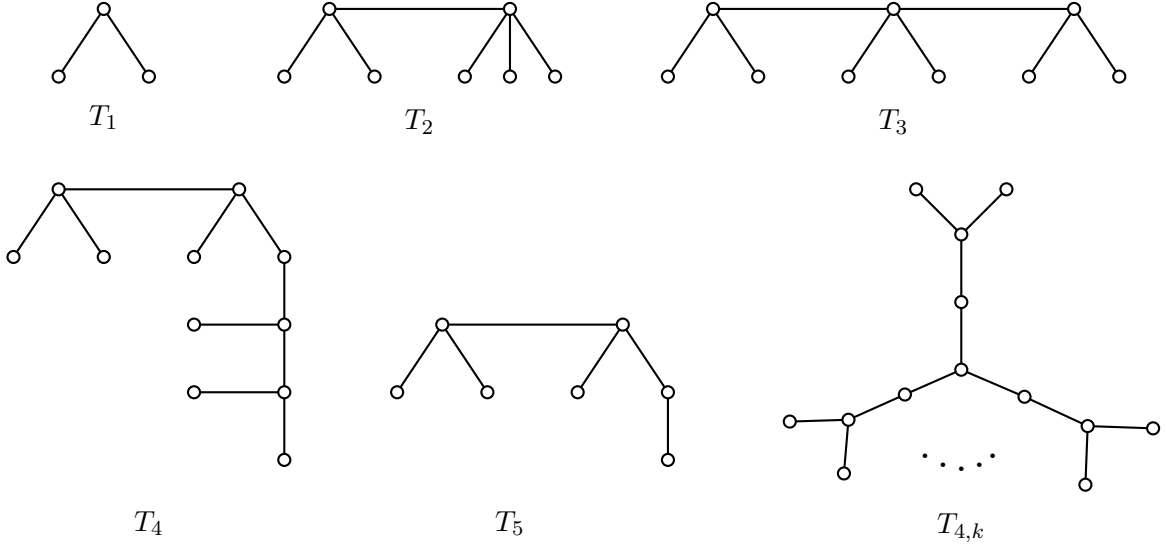
\begin{figure}[h]\vspace{-18mm}
\tikzstyle{every node}=[circle, draw, fill=white!, inner sep=0pt,minimum width=.16cm]
\begin{center}
\begin{tikzpicture}[thick,scale=.6]
  \draw(0,0) {

+(-13,-5) node{}
+(-11,-5) node{}
+(-12,-3.5) node{}
+(-13,-5) --+(-12,-3.5) --+(-11,-5)

+(-7,-3.5) node{}
+(-3,-3.5) node{}
+(-8,-5) node{}
+(-6,-5) node{}
+(-4,-5) node{}
+(-3,-5) node{}
+(-2,-5) node{}
+(-8,-5) --+(-7,-3.5) --+(-6,-5)
+(-4,-5) --+(-3,-3.5) --+(-2,-5)
+(-7,-3.5) --+(-3,-3.5) --+(-3,-5)

+(5.5,-3.5) node{}
+(9.5,-3.5) node{}
+(4.5,-5) node{}
+(6.5,-5) node{}
+(8.5,-5) node{}
+(10.5,-5) node{}
+(1.5,-3.5) node{}
+(0.5,-5) node{}
+(2.5,-5) node{}
+(1.5,-3.5) --+(5.5,-3.5) --+(9.5,-3.5)
+(0.5,-5) --+(1.5,-3.5) --+(2.5,-5)
+(4.5,-5) --+(5.5,-3.5) --+(6.5,-5)
+(8.5,-5) --+(9.5,-3.5) --+(10.5,-5)

+(-4.5,-10.5) node{}
+(-0.5,-10.5) node{}
+(-5.5,-12) node{}
+(-3.5,-12) node{}
+(-1.5,-12) node{}
+(0.5,-12) node{}
+(0.5,-13.5) node{}
+(-5.5,-12) --+(-4.5,-10.5) --+(-3.5,-12)
+(-1.5,-12) --+(-0.5,-10.5) --+(0.5,-12) --+(0.5,-13.5)
+(-4.5,-10.5) --+(-0.5,-10.5)

+(-13,-7.5) node{}
+(-9,-7.5) node{}
+(-14,-9) node{}
+(-12,-9) node{}
+(-10,-9) node{}
+(-8,-9) node{}
+(-8,-10.5) node{}
+(-10,-10.5) node{}
+(-8,-12) node{}
+(-8,-13.5) node{}
+(-10,-12) node{}
+(-14,-9) --+(-13,-7.5) --+(-12,-9)
+(-10,-9) --+(-9,-7.5) --+(-8,-9) --+(-8,-10.5)
+(-13,-7.5) --+(-9,-7.5)
+(-8,-10.5) --+(-8,-12) --+(-8,-13.5)
+(-8,-12) --+(-10,-12)
+(-8,-10.5) --+(-10,-10.5)

+(6,-7.5) node{}
+(8,-7.5) node{}
+(7,-8.5) node{}
+(7,-10) node{}
+(7,-11.5) node{}
+(6,-7.5) --+(7,-8.5) --+(8,-7.5)
+(7,-8.5) --+(7,-10) --+(7,-11.5)

+(5.75,-12.05) node{}
+(4.5,-12.61) node{}
+(3.2,-12.65) node{}
+(4.4,-13.8) node{}
+(7,-11.5) --+(5.75,-12.05) --+(4.5,-12.61)
+(3.2,-12.65) --+(4.5,-12.61) --+(4.4,-13.8)

+(8.4,-12.1) node{}
+(9.8,-12.75) node{}
+(9.75,-14.05) node{}
+(11.25,-12.8) node{}
+(7,-11.5) --+(8.4,-12.1) --+(9.8,-12.75)
+(9.75,-14.05) --+(9.8,-12.75) --+(11.25,-12.8)

+(7,-13.7) node[rectangle, draw=white!0, fill=white!100]{$\textbf{.}$}
+(6.6,-13.6) node[rectangle, draw=white!0, fill=white!100]{$\textbf{.}$}
+(7.4,-13.6) node[rectangle, draw=white!0, fill=white!100]{$\textbf{.}$}
+(7.7,-13.4) node[rectangle, draw=white!0, fill=white!100]{$\textbf{.}$}
+(6.2,-13.4) node[rectangle, draw=white!0, fill=white!100]{$\textbf{.}$}

+(7,-15) node[rectangle, draw=white!0, fill=white!100]{$T_{4,k}$}
+(-11,-14.9) node[rectangle, draw=white!0, fill=white!100]{$T_{4}$}
+(-3,-14.9) node[rectangle, draw=white!0, fill=white!100]{$T_{5}$}
+(5.5,-6) node[rectangle, draw=white!0, fill=white!100]{$T_{3}$}
+(-5,-6) node[rectangle, draw=white!0, fill=white!100]{$T_{2}$}
+(-12,-5.9) node[rectangle, draw=white!0, fill=white!100]{$T_{1}$}

};
\end{tikzpicture}
\end{center}\vspace{-5mm}
\caption{Family $\mathcal{J}$ of all trees $T$ with $\gamma_{rI}(T)=(|V(T)|+3)/2$. Note that $T_{4,k}$ is obtained from $k\geq1$ copies of $K_{1,3}$ by joining a new vertex to a leaf of any of them.}\label{fig1}
\end{figure}

\begin{theorem}\label{Chara}
For any tree $T$, $\gamma_{rI}(T)=(|V(T)|+3)/2$ if and only if $T\in \mathcal{J}$.
\end{theorem}
\begin{proof}
It is easy to check that $\gamma_{rI}(T_{i})=(|V(T_{i})|+3)/2$ for each $1\leq i\leq5$. Now consider the tree $T_{4,k}$ for some $k\geq1$. It is obtained from $k\geq1$ copies of the star $H_{i}=K_{1,3}$ on set of vertices $\{u_{i},v_{i},w_{i},x_{i}\}$ with central vertex $u_{i}$ by adding a new vertex $z$ and joining it to $x_{i}$, for all $1\leq i\leq k$. It is not difficult to see that $(f(u_{i}),f(v_{i}),f(w_{i}),f(x_{i}))=(0,1,1,0)$ for all $1\leq i\leq k$, and $f(z)=2$ defines a $\gamma_{rI}(T_{4,k})$-function with weight $2k+2=(|V(T_{4,k})|+3)/2$.

Conversely, let $\gamma_{rI}(T)=(|V(T)|+3)/2$. This implies that $n=|V(T)|\geq3$ and that $T\neq S_{2,2}$. We proceed by induction on the order $n\geq3$ of $T\neq S_{2,2}$. Clearly, $T=T_{1}=P_{3}\in \mathcal{J}$ when $n=3$. Moreover, it is readily checked that $T\in\{T_{1},T_{2},T_{4,1}\}$ when diam$(T)\leq3$. Hence, in what follows we may assume that diam$(T)\geq4$, which implies that $n\geq5$.

Let $T'\in \mathcal{J}$ for any tree $T'\neq S_{2,2}$ of order $3\leq n'<n$ for which $\gamma_{rI}(T')=(n'+3)/2$. Suppose now that $T\neq S_{2,2}$ is a tree of order $n$ for which $\gamma_{rI}(T)=(n+3)/2$. From now on, we make use of the notations given in the proof of Theorem \ref{Tree}. Again, we consider two cases depending on $f(u)$.\vspace{0.75mm}

\textit{Case 1.} $f(u)\geq1$. If $T'=T-v=S_{2,2}$, we have $T=T_{5}\in \mathcal{J}$. So, we assume that $T'\neq S_{2,2}$. On the other hand, $n(T')>3$ since diam$(T)\geq4$. In such a situation, the inequality chain (\ref{EQU1}) contradicts the fact that $\gamma_{rI}(T)=(n+3)/2$.\vspace{0.75mm}

\textit{Case 2.} $f(u)=0$. We have $f(w)=0$ as it was already mentioned in the proof of Theorem \ref{Tree}. Following the possibilities in the proof of Theorem \ref{Tree} we have two more cases.\vspace{0.75mm}

\textit{Subcase 2.1.} $N_{T}(w)\cap(V_{0}\setminus\{u\})=\emptyset$. Suppose first that $T''=T-V(T_{w})=S_{2,2}$. Similar to Subcase $2.1.2$ in the proof of Theorem \ref{Tree}, $wb\notin E(T)$ and we may assume that $wb_{2}\in E(T)$. Let deg$(w)=2$. If $\ell_{u}\geq2$, then $\gamma_{rI}(T)=5+\ell_{u}>(\ell_{u}+11)/2=(n+3)/2$. This is a contradiction. If $\ell_{u}=1$, then $\gamma_{rI}(T)=7>6=(n+3)/2$ which is again a contradiction. Therefore, deg$(w)\geq3$. If $\ell_{u}=1$, then $\gamma_{rI}(T)=\omega(f)\geq n-3>(n+3)/2$, a contradiction. Therefore, $\ell_{u}\geq2$. We then have $\gamma_{rI}(T)=\omega(f)=n-4\geq(n+3)/2$ with equality if and only if $n=11$. In such a situation $\ell_{u}=2$ and deg$(w)=3$, necessarily. Therefore, $T=T_{4}\in \mathcal{J}$.

We now consider the situation in which $n(T'')=2$. We first assume that deg$(w)=2$. If $\ell_{u}=1$, then $T=P_{5}$ with $\gamma_{rI}(P_{5})=5>(n+3)/2$ which is impossible. So, $\ell_{u}\geq2$. Then, $\gamma_{rI}(T)=\ell_{u}+3>(n+3)/2$ which is again impossible. Therefore, deg$(w)\geq3$. We have $\gamma_{rI}(T)=\omega(f)=n-1>(n+3)/2$ if $\ell_{u}=1$. So, $\ell_{u}\geq2$. In such a situation, we have $\gamma_{rI}(T)=n-2\geq(n+3)/2$ with equality if and only if $\ell_{u}=2$ and deg$(w)=3$. Therefore, $T=T_{5}\in \mathcal{J}$. So, we turn our attention to the situation in which $T''\neq S_{2,2}$ and $n(T'')\geq3$. We consider the following two possibilities.\vspace{0.75mm}

\textit{Subcase 2.1.1.} $\ell_{u}=1$. In this case, the inequality chain (\ref{EQU2}) implies that $\gamma_{rI}(T)>(n+3)/2$. This is a contradiction.\vspace{0.75mm}

\textit{Subcase 2.1.2.} $\ell_{u}\geq2$. Here both inequalities in (\ref{EQU3}) hold with equality, necessarily. This shows that $p+q=3$ and that $\gamma_{rI}(T'')=(n(T'')+3)/2$. Therefore $\ell_{u}=2$, deg$(w)=2$, and $T''\in \mathcal{J}$ by the induction hypothesis. Note that $T_{w}$ is isomorphic to the star $K_{1,3}$. Moreover, it is not difficult to check that a tree $T$ obtained from a copy of $K_{1,3}$ and a copy of $T_{i}\in\{T_{1},\cdots,T_{5}\}$ by joining $w$ to any vertex of $T_{i}$ does not satisfy $\gamma_{rI}(T)=(|V(T)|+3)/2$. Therefore, $T''=T_{4,k}$ for some $k\geq1$. Furthermore, the vertex $w$ must be necessarily adjacent to the vertex $z$ of $T_{4,k}$ in order that $T$ satisfies $\gamma_{rI}(T)=(|V(T)|+3)/2$. It is now clear that $T=T_{4,k+1}\in \mathcal{J}$.\vspace{0.75mm}

\textit{Subcase 2.2.} $N_{T}(w)\cap(V_{0}\setminus\{u\})\neq\emptyset$. We now distinguish the following two cases.\vspace{0.75mm}

\textit{Subcase 2.2.1.} $T'''=T-V(T_{u})=S_{2,2}$ and $u$ is adjacent to $b=w$. We have $\gamma_{rI}(T)=6>(n+3)/2$ when $\ell_{u}=1$. When $\ell_{u}\geq2$, we have $\gamma_{rI}(T)=4+\ell_{u}\geq(n+3)/2$ with equality if and only if $\ell_{u}=2$. This shows $T=T_{3}\in \mathcal{J}$.\vspace{0.75mm}

\textit{Subcase 2.2.2.} $T'''=S_{2,2}$ and $u$ is adjacent to $b_{2}=w$. In such a situation, $f(u)=f(w)=f(b)=0$. This is a contradiction.\vspace{0.75mm}

\textit{Subcase 2.2.3.} $T'''\neq S_{2,2}$. Notice that if $\ell_{u}=1$, then $\omega(f''')=\gamma_{rI}(T)-2$. Therefore,
\begin{equation*}
\gamma_{rI}(T)-2=\omega(f''')\geq \gamma_{rI}(T''')\geq\frac{n(T''')+3}{2}=\frac{n+1}{2}.
\end{equation*}
Therefore $\gamma_{rI}(T)>(n+1)/2$, a contradiction. Finally, $\omega(f''')=\gamma_{rI}(T)-\ell_{u}$ when $\ell_{u}\geq2$. Therefore,
\begin{equation*}
\gamma_{rI}(T)-\ell_{u}=\omega(f''')\geq \gamma_{rI}(T''')\geq\frac{n(T''')+3}{2}=\frac{n-\ell_{u}+2}{2}.
\end{equation*}
This ends up with the final contradiction $\gamma_{rI}(T)\geq(n+\ell_{u}+2)/2>(n+3)/2$.

The above discussion guarantees that $T\in \mathcal{J}=\{T_{1},\cdots,T_{5}\}\cup\{T_{4,k}\}$ for some $k\geq1$. This completes the proof.
\end{proof}

Some relations between RD number and RID number can be established based on the inherent properties of their concepts. For instance, we have the following realizability result.

\begin{proposition}\label{Prop1}
For any connected graph $G$, $\gamma_{r}(G)\leq \gamma_{rI}(G)\leq2\gamma_{r}(G)$. Furthermore, an ordered pair $(a,b)$ is realizable as the RD number and RID number for some nontrivial trees if and only if $2\leq a\leq b\leq2a$ and $(a,b)\neq(2,3)$.
\end{proposition}
\begin{proof}
Let $f=(V_{0},V_{1},V_{2})$ be a $\gamma_{rI}(G)$-function. Clearly, $V_{1}\cup V_{2}$ is an RD set in $G$. Therefore, $\gamma_{r}(G)\leq|V_{1}|+|V_{2}|\leq|V_{1}|+2|V_{2}|=\gamma_{rI}(G)$. In order to verify the upper bound, it suffices to consider the RID function $(V(G)\setminus S,\emptyset,S)$ for any $\gamma_{r}(G)$-set $S$.

Suppose that $2\leq a\leq b\leq2a$ and $(a,b)\neq(2,3)$. Let $b=2a$. Let $T$ be obtained from the star $K_{1,a-1}$ with central vertex $u$ by subdividing each edge twice. Note that $L(T)\cup\{u\}$ is the unique RD set of $T$ in which $L(T)$ is the set of leaves of $T$. So, $\gamma_{r}(T)=a$. On the other hand, $f=(V_{0},V_{1},V_{2})=(V(T)\setminus(L(T)\cup\{u\}),\emptyset,L(T)\cup\{u\})$ is an RID function with weight $\gamma_{rI}(T)=b$.

If $a=b$, then the star $K_{1,a-1}$ satisfies that $\gamma_{r}(K_{1,a-1})=\gamma_{rI}(K_{1,a-1})=a$. So, we may assume that $a<b<2a$. If $a=2$, then $b=3$ which is impossible. Therefore, $a\geq3$. We begin with the star $K_{1,a}$ with $V(K_{1,a})=\{u,v_{1},\cdots,v_{a}\}$ in which $u$ is the central vertex. We add $b-a$ new vertices $w_{i}$ and edges $v_{i}w_{i}$ for $1\leq i\leq b-a$. Let $T'$ be the resulting tree. Clearly, the set of leaves $L(T')$ is the unique RD set in $T'$ of cardinality $\gamma_{r}(T')=a$. It is now easy to see that $f'=(V_{0},V_{1},V_{2})=(\{u,v_{1}\},V(T)\setminus\{u,v_{1},w_{1}\},\{w_{1}\})$ defines an RID function with weight $\omega(f')=\gamma_{rI}(T')=b$.

Conversely, suppose that $\gamma_{r}(T)=a$ and $\gamma_{rI}(T)=b$ for some nontrivial tree $T$. The bounds in the theorem imply that $a\leq b\leq2a$. Also, $a\geq2$ since $T$ is a nontrivial tree. Suppose now that $a=2$ and $T\neq P_{2}$. Let $S=\{x,y\}$ be a $\gamma_{r}(T)$-set. Since $T\neq P_{2}$, it follows that $V(T)\setminus S\neq \emptyset$. This shows that $T[V(T)\setminus S]$ is a forest consisting of $t\geq1$ nontrivial components $T_{1},\cdots,T_{t}$. Since $T$ is a tree, it follows that $|V(T_{i})|\leq2$ for all $1\leq i\leq t$. Therefore, any component $T_{i}$ is isomorphic to the path $P_{2}$. Moreover, any vertex of $T_{i}$ has precisely one neighbor in $S$ since $T$ is a tree. If $t\geq2$, then $T$ contains the cycle $C_{6}$ as a subgraph, a contradiction. Therefore, $t=1$. This implies that $T[V(T)\setminus S]\cong P_{2}$ and so, $T\cong P_{4}$. In fact, we have shown that $\gamma_{r}(T)=2$ if and only if $T\in\{P_{2},P_{4}\}$. Therefore, the case $(a,b)=(2,3)$ is impossible.
\end{proof}


\section{Graphs with small or large RID numbers}

\subsection{Graphs $G$ with $\gamma_{rI}(G)=i$ when $i\in\{2,3\}$}

Let $H$ be a complete bipartite graph of order $n\geq3$ with partite sets $X$ and $Y$ with $|X|\leq|Y|$ such that $|X|\in\{1,2\}$ and $|Y|\geq2$. Let $\Omega$ be the family of all graphs $G$ obtained from $H$ by adding some edges among the vertices in $Y$ such that $\delta(G[Y])\geq1$.

\begin{theorem}\label{Prop2}
For any connected graph $G$ of order $n\geq2$, $\gamma_{rI}(G)=2$ if and only if $G\in \Omega\cup\{P_{2}\}$.
\end{theorem}
\begin{proof}
It is routine to check that $\gamma_{rI}(G)=2$ if $G\in \Omega\cup\{P_{2}\}$.

Conversely, suppose that $\gamma_{rI}(G)=2$ and $G\neq P_{2}$. Let $f=(V_{0},V_{1},V_{2})$ be a $\gamma_{rI}(G)$-function. We consider two cases depending on $V_{2}$.\vspace{0.75mm}

\textit{Case 1.} $V_{2}\neq\emptyset$. In such a case, there is a unique vertex $v$ with $f(v)=2$, and the other $n-1\geq2$ vertices are assigned $0$ under $f$. Moreover, all vertices in $V(G)\setminus\{v\}$ are adjacent to $v$ and there is no isolated vertex in $G[V(G)\setminus\{v\}]$ by the definition of $f$. We observe that $G\in \Omega$ by taking $X$ and $Y$ as $\{v\}$ and $V(G)\setminus\{v\}$, respectively.\vspace{0.75mm}

\textit{Case 2.} $V_{2}=\emptyset$. So, there are two vertices $u$ and $v$ with $f(u)=f(v)=1$, and the other vertices are assigned $0$ under $f$. We consider two other possibilities.\vspace{0.75mm}

\textit{Subcase 2.1.} Let $uv\notin E(G)$. We take $X$ and $Y$ as $\{u,v\}$ and $V(G)\setminus\{u,v\}$, respectively. We have $|Y|\geq2$ since $G$ cannot be $P_{3}$. We now deduce that $G\in \Omega$ from the fact that each vertex in $Y$ is adjacent to both vertices in $X$ and another vertex in $Y$.\vspace{0.75mm}

\textit{Subcase 2.2.} Let $uv\in E(G)$. In such a case, setting $X=\{u\}$ and $Y=V(G)\setminus\{u\}$ satisfies $G\in \Omega$. This completes the proof.
\end{proof}

Let $\Psi$ consist of all graphs $G$ satisfying one of the following statements $(i)$ and $(ii)$.\vspace{0.75mm}

$(i)$ $\Delta(G)=n-1$ and $G$ has a unique vertex of degree one.\vspace{0.75mm}

$(ii)$ $G$ is obtained from a graph $H$ with $\delta(H)\geq1$ by adding two vertices $x$ and $y$ and adding edges with one end point in $\{x,y\}$ and the other in $V(H)$ such that $N_{G}(x)=V(H)$ and that $1\leq$ deg$(y)\leq|V(H)|-1$.\vspace{0.75mm}

Finally, suppose that $H$ and $K$ be two graphs with $|V(H)|=3$ and $\delta(K)\geq1$. Then, $G$ is obtained from joining each vertex of $K$ to at least two vertices of $H$ so that the resulting graph is connected. Let $\Theta$ be the family of all resulting graphs $G$.

\begin{theorem}
Let $G$ be a connected graph. Then, $\gamma_{rI}(G)=3$ if and only if $G\in \Psi \cup(\Theta\setminus \Omega)\cup\{P_{3}\}$.
\end{theorem}
\begin{proof}
Let $G\in \Psi$. Assigning $2$ to the vertex of maximum degree $n-1$, $1$ to the unique vertex of degree one and $0$ to the other vertices defines an RID function with weight $\gamma_{rI}(G)=3$ when $G$ satisfies $(i)$. Let $G$ satisfy $(ii)$. Then, $(f(x),f(y))=(2,1)$ and $f(v)=0$ for the other vertices is an RID function with weight $\gamma_{rI}(G)=3$. Let $G\in \Theta\setminus \Omega$. The assignment $g(u)=1$ for each $u\in V(H)$, and $g(v)=0$ for each $v\in V(K)$ is an RID function with weight $3$. So, $\gamma_{rI}(G)\leq3$. Moreover, $\gamma_{rI}(G)>2$ since $G\notin \Omega$. Therefore, $\gamma_{rI}(G)=3$.

Conversely, let $f=(V_{0},V_{1},V_{2})$ be a $\gamma_{rI}(G)$-function. We deal with two cases depending on the equality $\gamma_{rI}(G)=|V_{1}|+2|V_{2}|=3$.\vspace{0.75mm}

\textit{Case 1.} $(|V_{1}|,|V_{2}|)=(1,1)$. Let $V_{1}=\{y\}$ and $V_{2}=\{x\}$. Since $V(G)\setminus\{x,y\}=V_{0}$, every vertex in this subset has a neighbor in $\{x,y\}$. Moreover, each such a vertex is adjacent to $x$, necessarily. Therefore, deg$(x)\geq n-2$. If deg$(x)=n-1$, then $y$ is a vertex of degree one in $G$, otherwise $\gamma_{rI}(G)=2$. Moreover, if there exists a vertex $z\neq y$ of degree one, then $f(z)=1$ which is impossible. So, $G$ satisfies ($i$). We now assume that deg$(x)=n-2$. Since $G$ is connected, it follows that deg$(y)\geq1$. Moreover, $N(y)\subset V(G)\setminus\{x,y\}$, for otherwise $\gamma_{rI}(G)=2$. We now deduce that $G$ satisfies ($ii$) by using $G[V_{0}]$ instead of $H$ in ($ii$). We have shown that $G\in \Psi$ in this case.\vspace{0.75mm}

\textit{Case 2.} Suppose that $(|V_{1}|,|V_{2}|)=(3,0)$ and $G\neq P_{3}$. In such a situation, it is readily seen that $G$ is obtained from two graphs $H=G[V_{1}]$ and $K=G[V_{0}]$. That $G$ is a member of $\Theta$ follows by the definition of $f$. Moreover, $G\notin \Omega$ as $\gamma_{rI}(G)\neq2$.
\end{proof}

\subsection{Graphs $G$ with $\gamma_{rI}(G)=i$ when $i\in\{n-1,n\}$}

\begin{theorem}
Let $G$ be a connected graph of order $n$. Then, $\gamma_{rI}(G)=n$ if and only if $G\in\{K_{1},K_{1,n-1} (n\geq2),C_{4},C_{5},P_{4},P_{5},P_{6}\}$.
\end{theorem}
\begin{proof}
Let $\gamma_{rI}(G)=n$. We distinguish two cases depending on the existence of cycles in $G$.\vspace{0.75mm}

\textit{Case 1.} Suppose that $G=T$ is a tree. If $T=K_{1}$ or $T=K_{1,n-1}$ for $n\geq2$, then we are done. So, we may assume that $T$ is neither a trivial tree nor a star. We claim that $\Delta(T)\leq2$. Suppose to the contrary that there exists a vertex $v$ with deg$(v)\geq3$. Since $T$ is not a star, it follows that the vertex $v$ is adjacent to a non-leaf vertex $u$. Let $w$ be a neighbor of $u$ different from $v$. Then, the assignment $(f(u),f(v),f(w))=(0,0,2)$ and $f(x)=1$ for the other vertices $x$ defines an RID function with weight $\omega(f)=n-1$, a contradiction. Therefore, $\Delta(T)=2$, and so $T$ is isomorphic to a path on $n\geq4$ vertices. Consider the path $P_{n}:x_{1}x_{2}\cdots x_{n}$ for $n\geq7$. It is easy to see that $g(x_{1})=g(x_{4})=g(x_{7})=2$, $g(x_{2})=g(x_{3})=g(x_{5})=g(x_{6})=0$ and $g(x_{i})=1$ for $i\geq8$ (if any) is an RID function of $P_{n}$ with weight $\omega(g)=n-1$, a contradiction. The above discussion shows that $T\in\{K_{1},K_{1,n-1} (n\geq2),P_{4},P_{5},P_{6}\}$.\vspace{0.75mm}

\textit{Case 2.} Suppose that $G$ contains at least one cycle. It is easily seen that $G$ is triangle-free, otherwise $\gamma_{rI}(G)<n$. Let $P_{t}$ be a longest path in $G$. An argument similar to what presented in Case $1$ (related to $P_{n}$ for $n\geq7$) implies that $t\leq6$. We now consider a $k$-cycle $C_{k}:x_{1}x_{2}\cdots x_{k}x_{1}$ in $G$. Since $G$ does not contain any path on $k\geq7$ vertices as a subgraph, it follows that $k\leq6$. If $k=6$, then $h(x_{1})=h(x_{4})=2$, $h(x_{2})=h(x_{3})=h(x_{5})=h(x_{6})=0$ and $h(x)=1$ for the other vertices $x$ (if any) defines an RID function with weight $n-2$, a contradiction. Therefore, $k\in\{4,5\}$. Let $k=5$. Suppose that $V(C_{5})\subset V(G)$ and that $v\in V(G)\setminus V(C_{5})$ is adjacent to a vertex
of $C_{5}$, say $x_{1}$. Then, the assignment $(h(x_{1}),h(x_{5}),h(x_{4}))=(0,0,2)$ and $h(x)=1$ for any other vertex $x$ defines an RID function with weight $n-1$. This is a contradiction. Therefore, $V(C_{5})=V(G)$. On the other hand, there is no chord between any to vertices in $V(C_{5})$ since $G$ is triangle-free. Thus, $G=C_{5}$. A similar argument implies that $G=C_{4}$ when $k=4$. In such a case, we have proved that $G\in\{C_{4},C_{5}\}$.

Conversely, it is easily verified that $\gamma_{rI}(G)=n$ if $G\in\{K_{1},K_{1,n-1} (n\geq2),C_{4},C_{5},P_{4},P_{5},P_{6}\}$.
\end{proof}

In order to characterize the family of all connected graphs $G$ with $\gamma_{rI}(G)=|V(G)|-1$, we shall need the following helpful lemma.

\begin{lemma}\label{lem1}
Let $G$ be a connected graph of order $n$. Then, $\gamma_{rI}(G)\le n-2$ if one of the following statements holds.\vspace{0.5mm}\\
\emph{(1)} There exist two adjacent vertices $u,v\in V(G)$ such that $deg(u),deg(v)\geq 3$.\vspace{0.5mm}\\
\emph{(2)} $diam(G)\geq 9$.\vspace{0.5mm}\\
\emph{(3)} There exist two vertices $u,v\in V(G)$ with $d(u,v)=4$, $deg(u)\geq 3$ and $deg(v)\geq 2$.\vspace{0.5mm}\\
\emph{(4)} There exist at least three edge disjoint paths $P_l$, $P_k$ and $P_m$ \emph{(}$k,l,m\ge 4$\emph{)} which have precisely one end point $u$ in common.\vspace{0.5mm}\\
\emph{(5)} There exists a subgraph $G'$ obtained from $P_7:v_1v_2v_3v_4v_5v_6v_7$ by joining a vertex $v_{3}'$ to $v_3$ and a vertex $v_{5}'$ to $v_5$.
\end{lemma}
\begin{proof}
(1) The assignment $f(u)=f(v)=0$ and $f(x)=1$ for each $x\in V(G)\setminus\{u,v\}$ is an RID function. So, $\gamma_{rI}(G)\leq w(f)=n-2$.

(2) Let $d(u,v)=9$ and $uw_1w_2...w_8v$ be a $u,v$-path. Note that $f(u)=f(w_3)=f(w_6)=f(v)=2$, $f(w_1)=f(w_2)=f(w_4)=f(w_5)=f(w_7)=f(w_8)=0$ and $f(x)=1$ for any other vertex $x$ is an RID function of $G$ with weight $n-2$. So, we have $\gamma_{rI}(G)\leq n-2$.

(3) Let $uw_1w_2w_3v$ be a $u,v$-path of length four and $t \in N(v)\setminus \{w_3\}$. Then, $f(u)=f(w_1)=f(w_3)=f(v)=0$, $f(w_2)=f(t)=2$ and $f(x)=1$ for $x\in V(G)\setminus\{u,w_1,w_2,w_3,v,t\}$ is an RID function of $G$. Therefore, $\gamma_{rI}(G)\leq w(f)=n-2$.

(4) Let $P_k:x_1x_2x_3\cdots x_k$, $P_l:y_1y_2y_3\cdots y_l$, $P_m:z_1z_2z_3\cdots z_m$ in which $u=x_1=y_1=z_1$. We assign $2$ to $u,x_4,y_4,z_4$, $0$ to $x_2, x_3, y_2, y_3, z_2, z_3$ and $1$ to the other vertices. This gives us an RID function of $G$ with weight $n-2$. So, $\gamma_{rI}(G)\leq w(f)=n-2$.

(5) Assigning $2$ to $v_1$ and $v_7$, $0$ to $v_2,v_3,v_5$ and $v_6$, and $1$ to the other vertices gives us an RID function of $G$. Therefore, $\gamma_{rI}(G)\leq w(f)=n-2$.
\end{proof}

We make use of the family $\mathcal{G}$ depicted in Figure \ref{fig2} so as to give the characterization of all connected graphs  for which the RID number equals the order minus one. We need to mention some supplementary explanations concerning this family.\\
\textit{1.} Both vertices $u$ and $v$ have degree at least three in $G_1$.\\
\textit{2.} There exist $t\geq1$ leaves at distance at most two from the cycle $C_5$ in $G_2$.\\
\textit{3.} There exist exactly one leaf at distance three and $t\geq0$ leaves at distance at most two from $C_5$ in $G_3$.\\
\textit{4.} There exist exactly one leaf at distance four and $t\geq0$ leaves at distance at most two from $C_5$ in $G_4$.\\
\textit{5.} There are exactly two cycles $C_5$ and $t\geq0$ leaves at distance at most two from them in $G_5$.\\
\textit{6.} There are $t\geq1$ leaves at distance at most two from $C_3$ in $G_{11}$.\\
\textit{7.} There are $t\geq1$ leaves at distance at most two from $u$ in $T_1$, $T_2$ and $T_3$.\\
\textit{8.} There are $t\geq1$ leaves different from $v$ at distance at most two from $u$ in $T_4$.\\
\textit{9.} There are $t\geq1$ leaves different from $u$ and $v$ in $T_5$.\\
\textit{10.} In $T_6$, there are $t\geq1$ leaves different from $u$ at distance at most two from $v$ and there is at least one leaf different from $x$ adjacent to $w$.\\
\textit{11.} There are $t\geq1$ leaves different from $v$ and $w$ at distance at most two from $u$ in $T_7$ and $T_8$.\\
\textit{12.} There are $t\geq1$ leaves different from $u$ and $v$ in $T_9$.\\\\\\

\begin{figure}[h]\vspace{-18mm}
\tikzstyle{every node}=[circle, draw, fill=white!, inner sep=0pt,minimum width=.16cm]
\begin{center}
\begin{tikzpicture}[thick,scale=.6]
  \draw(0,0) {

+(-13.5,-4) node{}
+(-10.5,-4) node{}
+(-12,-3) node{}
+(-13,-5.6) node{}
+(-11,-5.6) node{}
+(-12.75,-2) node{}
+(-11.25,-2) node{}
+(-14.2,-6.4) node{}
+(-12.8,-6.9) node{}
+(-13,-5.6) --+(-13.5,-4) --+(-12,-3) --+(-10.5,-4) --+(-11,-5.6) --+(-13,-5.6)
+(-12.75,-2) --+(-12,-3) --+(-11.25,-2)
+(-14.2,-6.4) --+(-13,-5.6) --+(-12.8,-6.9)

+(-12.2,-2.25) node[rectangle, draw=white!0, fill=white!100]{$\textbf{.}$}
+(-12,-2.25) node[rectangle, draw=white!0, fill=white!100]{$\textbf{.}$}
+(-11.8,-2.25) node[rectangle, draw=white!0, fill=white!100]{$\textbf{.}$}
+(-13.6,-6.4) node[rectangle, draw=white!0, fill=white!100]{$\textbf{.}$}
+(-13.4,-6.5) node[rectangle, draw=white!0, fill=white!100]{$\textbf{.}$}
+(-13.2,-6.6) node[rectangle, draw=white!0, fill=white!100]{$\textbf{.}$}

+(-12.5,-2.95) node[rectangle, draw=white!0, fill=white!100]{$u$}
+(-13.4,-5.5) node[rectangle, draw=white!0, fill=white!100]{$v$}

+(-11.9,-4.5) node[rectangle, draw=white!0, fill=white!100]{$G_{1}$}


+(-8.5,-4.5) node{}
+(-5.5,-4.5) node{}
+(-7,-3.5) node{}
+(-8,-6.1) node{}
+(-6,-6.1) node{}
+(-9.5,-2.75) node{}
+(-8.25,-2.5) node{}
+(-7,-2.5) node{}
+(-7,-1.5) node{}
+(-6.35,-2.6) node{}
+(-5.6,-1.55) node{}
+(-8,-6.1) --+(-8.5,-4.5) --+(-7,-3.5) --+(-5.5,-4.5) --+(-6,-6.1) --+(-8,-6.1)
+(-9.5,-2.75) --+(-7,-3.5) --+(-8.25,-2.5)
+(-7,-3.5) --+(-7,-2.5) --+(-7,-1.5)
+(-7,-3.5) --+(-6.35,-2.6) --+(-5.6,-1.55)

+(-8.75,-2.75) node[rectangle, draw=white!0, fill=white!100]{$\textbf{.}$}
+(-8.55,-2.75) node[rectangle, draw=white!0, fill=white!100]{$\textbf{.}$}
+(-8.35,-2.75) node[rectangle, draw=white!0, fill=white!100]{$\textbf{.}$}
+(-6.6,-1.9) node[rectangle, draw=white!0, fill=white!100]{$\textbf{.}$}
+(-6.4,-1.9) node[rectangle, draw=white!0, fill=white!100]{$\textbf{.}$}
+(-6.2,-1.9) node[rectangle, draw=white!0, fill=white!100]{$\textbf{.}$}

+(-6.9,-5) node[rectangle, draw=white!0, fill=white!100]{$G_{2}$}


+(-3.5,-4.5) node{}
+(-0.5,-4.5) node{}
+(-2,-3.5) node{}
+(-3,-6.1) node{}
+(-1,-6.1) node{}
+(-4.5,-2.75) node{}
+(-3.25,-2.5) node{}
+(-2,-2.5) node{}
+(-2,-1.5) node{}
+(-1.35,-2.6) node{}
+(-0.6,-1.55) node{}
+(-0.8,-3) node{}
+(0.2,-2.55) node{}
+(1.2,-2.05) node{}
+(-3,-6.1) --+(-3.5,-4.5) --+(-2,-3.5) --+(-0.5,-4.5) --+(-1,-6.1) --+(-3,-6.1)
+(-4.5,-2.75) --+(-2,-3.5) --+(-3.25,-2.5)
+(-2,-3.5) --+(-2,-2.5) --+(-2,-1.5)
+(-2,-3.5) --+(-1.35,-2.6) --+(-0.6,-1.55)
+(-2,-3.5) --+(-0.8,-3) --+(0.2,-2.55) --+(1.2,-2.05)

+(-3.75,-2.75) node[rectangle, draw=white!0, fill=white!100]{$\textbf{.}$}
+(-3.55,-2.75) node[rectangle, draw=white!0, fill=white!100]{$\textbf{.}$}
+(-3.35,-2.75) node[rectangle, draw=white!0, fill=white!100]{$\textbf{.}$}
+(-1.6,-1.9) node[rectangle, draw=white!0, fill=white!100]{$\textbf{.}$}
+(-1.4,-1.9) node[rectangle, draw=white!0, fill=white!100]{$\textbf{.}$}
+(-1.2,-1.9) node[rectangle, draw=white!0, fill=white!100]{$\textbf{.}$}

+(-1.9,-5) node[rectangle, draw=white!0, fill=white!100]{$G_{3}$}


+(2.5,-4.5) node{}
+(5.5,-4.5) node{}
+(4,-3.5) node{}
+(3,-6.1) node{}
+(5,-6.1) node{}
+(1.5,-2.75) node{}
+(2.75,-2.5) node{}
+(4,-2.5) node{}
+(4,-1.5) node{}
+(4.65,-2.6) node{}
+(5.4,-1.55) node{}
+(5.2,-3) node{}
+(6.2,-2.55) node{}
+(7.2,-2.05) node{}
+(8.25,-1.55) node{}
+(3,-6.1) --+(2.5,-4.5) --+(4,-3.5) --+(5.5,-4.5) --+(5,-6.1) --+(3,-6.1)
+(1.5,-2.75) --+(4,-3.5) --+(2.75,-2.5)
+(4,-3.5) --+(4,-2.5) --+(4,-1.5)
+(4,-3.5) --+(4.65,-2.6) --+(5.4,-1.55)
+(4,-3.5) --+(5.2,-3) --+(6.2,-2.55) --+(7.2,-2.05) --+(8.25,-1.55)

+(2.25,-2.75) node[rectangle, draw=white!0, fill=white!100]{$\textbf{.}$}
+(2.45,-2.75) node[rectangle, draw=white!0, fill=white!100]{$\textbf{.}$}
+(2.65,-2.75) node[rectangle, draw=white!0, fill=white!100]{$\textbf{.}$}
+(4.4,-1.9) node[rectangle, draw=white!0, fill=white!100]{$\textbf{.}$}
+(4.6,-1.9) node[rectangle, draw=white!0, fill=white!100]{$\textbf{.}$}
+(4.8,-1.9) node[rectangle, draw=white!0, fill=white!100]{$\textbf{.}$}

+(4.1,-5) node[rectangle, draw=white!0, fill=white!100]{$G_{4}$}

+(-13,-10.5) node{}
+(-10,-10.5) node{}
+(-11.5,-9.5) node{}
+(-12.5,-12.1) node{}
+(-10.5,-12.1) node{}
+(-14,-8.75) node{}
+(-12.75,-8.5) node{}
+(-11.5,-8.5) node{}
+(-11.5,-7.5) node{}
+(-10.85,-8.6) node{}
+(-10.1,-7.55) node{}
+(-10.3,-9.4) node{}
+(-9.3,-8.95) node{}
+(-8.3,-8.35) node{}
+(-7.7,-7.35) node{}
+(-12.5,-12.1) --+(-13,-10.5) --+(-11.5,-9.5) --+(-10,-10.5) --+(-10.5,-12.1) --+(-12.5,-12.1)
+(-14,-8.75) --+(-11.5,-9.5) --+(-12.75,-8.5)
+(-11.5,-9.5) --+(-11.5,-8.5) --+(-11.5,-7.5)
+(-11.5,-9.5) --+(-10.85,-8.6) --+(-10.1,-7.55)
+(-11.5,-9.5) --+(-10.3,-9.4) --+(-9.3,-8.95) --+(-8.3,-8.35) --+(-7.7,-7.35)
+(-11.5,-9.5) --+(-7.7,-7.35)

+(-13.25,-8.75) node[rectangle, draw=white!0, fill=white!100]{$\textbf{.}$}
+(-13.05,-8.75) node[rectangle, draw=white!0, fill=white!100]{$\textbf{.}$}
+(-12.85,-8.75) node[rectangle, draw=white!0, fill=white!100]{$\textbf{.}$}
+(-11.1,-7.9) node[rectangle, draw=white!0, fill=white!100]{$\textbf{.}$}
+(-10.9,-7.9) node[rectangle, draw=white!0, fill=white!100]{$\textbf{.}$}
+(-10.7,-7.9) node[rectangle, draw=white!0, fill=white!100]{$\textbf{.}$}

+(-11.4,-11) node[rectangle, draw=white!0, fill=white!100]{$G_{5}$}

+(-7,-9) node{}
+(-8,-10) node{}
+(-6,-10) node{}
+(-7,-11) node{}

+(-7,-7.75) node{}
+(-7,-12.25) node{}

+(-7,-9) --+(-8,-10) --+(-7,-11) --+(-6,-10) --+(-7,-9)
+(-7,-7.75) --+(-7,-9)
+(-7,-12.25) --+(-7,-11)

+(-7,-10) node[rectangle, draw=white!0, fill=white!100]{$G_{6}$}


+(-4,-9.5) node{}
+(-5,-10.5) node{}
+(-3,-10.5) node{}
+(-4,-11.5) node{}

+(-4.8,-12.3) node{}

+(-3.2,-8.7) node{}
+(-2.2,-7.7) node{}

+(-4,-9.5) --+(-5,-10.5) --+(-4,-11.5) --+(-3,-10.5) --+(-4,-9.5)
+(-4,-11.5) --+(-4.8,-12.3)
+(-4,-9.5) --+(-3.2,-8.7) --+(-2.2,-7.7)

+(-4,-10.4) node[rectangle, draw=white!0, fill=white!100]{$G_{7}$}


+(-1,-10) node{}
+(-2,-11) node{}
+(0,-11) node{}
+(-1,-12) node{}

+(-0.2,-9.2) node{}
+(0.7,-8.3) node{}
+(1.6,-7.4) node{}

+(-1,-10) --+(-2,-11) --+(-1,-12) --+(0,-11) --+(-1,-10)
+(-1,-10) --+(-0.2,-9.2) --+(0.7,-8.3)
+(0.7,-8.3) --+(1.6,-7.4)

+(-1,-10.9) node[rectangle, draw=white!0, fill=white!100]{$G_{8}$}


+(2,-10) node{}
+(1,-11) node{}
+(3,-11) node{}
+(2,-12) node{}

+(2.8,-9.2) node{}
+(3.7,-8.3) node{}

+(2,-10) --+(1,-11) --+(2,-12) --+(3,-11) --+(2,-10)
+(2,-10) --+(2.8,-9.2) --+(3.7,-8.3)

+(2,-11) node[rectangle, draw=white!0, fill=white!100]{$G_{9}$}


+(5,-9.8) node{}
+(4,-10.8) node{}
+(6,-10.8) node{}
+(5,-11.8) node{}

+(5,-8.55) node{}

+(5,-9.8) --+(4,-10.8) --+(5,-11.8) --+(6,-10.8) --+(5,-9.8)
+(5,-8.55) --+(5,-9.8)

+(5,-10.8) node[rectangle, draw=white!0, fill=white!100]{$G_{10}$}


+(-14,-16.95) node{}
+(-11,-16.95) node{}
+(-12.5,-15.45) node{}
+(-15,-14.7) node{}
+(-13.75,-14.45) node{}
+(-12.5,-14.45) node{}
+(-12.5,-13.45) node{}
+(-11.85,-14.55) node{}
+(-11.1,-13.5) node{}

+(-14,-16.95) --+(-12.5,-15.45) --+(-11,-16.95)
+(-15,-14.7) --+(-12.5,-15.45) --+(-13.75,-14.45)
+(-12.5,-15.45) --+(-12.5,-14.45) --+(-12.5,-13.45)
+(-12.5,-15.45) --+(-11.85,-14.55) --+(-11.1,-13.5)
+(-14,-16.95) --+(-11,-16.95)

+(-14.25,-14.7) node[rectangle, draw=white!0, fill=white!100]{$\textbf{.}$}
+(-14.05,-14.7) node[rectangle, draw=white!0, fill=white!100]{$\textbf{.}$}
+(-13.85,-14.7) node[rectangle, draw=white!0, fill=white!100]{$\textbf{.}$}
+(-12.1,-13.85) node[rectangle, draw=white!0, fill=white!100]{$\textbf{.}$}
+(-11.9,-13.85) node[rectangle, draw=white!0, fill=white!100]{$\textbf{.}$}
+(-11.7,-13.85) node[rectangle, draw=white!0, fill=white!100]{$\textbf{.}$}

+(-12.4,-16.4) node[rectangle, draw=white!0, fill=white!100]{$G_{11}$}


+(-7,-14.5) node{}
+(-7.8,-15.2) node{}
+(-8.55,-15.85) node{}
+(-9.25,-16.45) node{}
+(-10,-17.1) node{}
+(-7,-14.5) --+(-7.8,-15.2) --+(-8.55,-15.85) --+(-9.25,-16.45) --+(-10,-17.1)

+(-6.2,-15.2) node{}
+(-5.45,-15.85) node{}
+(-4.75,-16.45) node{}
+(-4,-17.1) node{}
+(-7,-14.5) --+(-6.2,-15.2) --+(-5.45,-15.85) --+(-4.75,-16.45) --+(-4,-17.1)

+(-7.75,-13.5) node{}
+(-6.25,-13.5) node{}
+(-7.75,-13.5) --+(-7,-14.5) --+(-6.25,-13.5)

+(-7.4,-15.6) node{}
+(-7.8,-16.5) node{}
+(-6.6,-15.6) node{}
+(-6.2,-16.5) node{}
+(-7.8,-16.5) --+(-7.4,-15.6) --+(-7,-14.5) --+(-6.6,-15.6) --+(-6.2,-16.5)

+(-7.2,-13.7) node[rectangle, draw=white!0, fill=white!100]{$\textbf{.}$}
+(-7,-13.7) node[rectangle, draw=white!0, fill=white!100]{$\textbf{.}$}
+(-6.8,-13.7) node[rectangle, draw=white!0, fill=white!100]{$\textbf{.}$}
+(-7.2,-16.2) node[rectangle, draw=white!0, fill=white!100]{$\textbf{.}$}
+(-7,-16.2) node[rectangle, draw=white!0, fill=white!100]{$\textbf{.}$}
+(-6.8,-16.2) node[rectangle, draw=white!0, fill=white!100]{$\textbf{.}$}

+(-7.5,-14.5) node[rectangle, draw=white!0, fill=white!100]{$u$}

+(-7,-17.1) node[rectangle, draw=white!0, fill=white!100]{$T_{1}$}


+(0,-14.5) node{}
+(-0.8,-15.2) node{}
+(-1.55,-15.85) node{}
+(-2.25,-16.45) node{}
+(-3,-17.1) node{}
+(0,-14.5) --+(-0.8,-15.2) --+(-1.55,-15.85) --+(-2.25,-16.45) --+(-3,-17.1)

+(0.8,-15.2) node{}
+(1.55,-15.85) node{}
+(2.25,-16.45) node{}
+(0,-14.5) --+(0.8,-15.2) --+(1.55,-15.85) --+(2.25,-16.45)

+(-0.75,-13.5) node{}
+(0.75,-13.5) node{}
+(-0.75,-13.5) --+(0,-14.5) --+(0.75,-13.5)

+(-0.4,-15.6) node{}
+(-0.8,-16.5) node{}
+(0.4,-15.6) node{}
+(0.8,-16.5) node{}
+(-0.8,-16.5) --+(-0.4,-15.6) --+(0,-14.5) --+(0.4,-15.6) --+(0.8,-16.5)

+(-0.2,-13.7) node[rectangle, draw=white!0, fill=white!100]{$\textbf{.}$}
+(0,-13.7) node[rectangle, draw=white!0, fill=white!100]{$\textbf{.}$}
+(0.2,-13.7) node[rectangle, draw=white!0, fill=white!100]{$\textbf{.}$}
+(-0.2,-16.2) node[rectangle, draw=white!0, fill=white!100]{$\textbf{.}$}
+(0,-16.2) node[rectangle, draw=white!0, fill=white!100]{$\textbf{.}$}
+(0.2,-16.2) node[rectangle, draw=white!0, fill=white!100]{$\textbf{.}$}

+(-0.5,-14.5) node[rectangle, draw=white!0, fill=white!100]{$u$}

+(0,-17.3) node[rectangle, draw=white!0, fill=white!100]{$T_{2}$}


+(5.5,-14.5) node{}
+(4.7,-15.2) node{}
+(3.95,-15.85) node{}
+(3.25,-16.45) node{}
+(5.5,-14.5) --+(4.7,-15.2) --+(3.95,-15.85) --+(3.25,-16.45)

+(6.3,-15.2) node{}
+(7.05,-15.85) node{}
+(7.75,-16.45) node{}
+(5.5,-14.5) --+(6.3,-15.2) --+(7.05,-15.85) --+(7.75,-16.45)

+(4.75,-13.5) node{}
+(6.25,-13.5) node{}
+(4.75,-13.5) --+(5.5,-14.5) --+(6.25,-13.5)

+(5.1,-15.6) node{}
+(4.7,-16.5) node{}
+(5.9,-15.6) node{}
+(6.3,-16.5) node{}
+(4.7,-16.5) --+(5.1,-15.6) --+(5.5,-14.5) --+(5.9,-15.6) --+(6.3,-16.5)

+(5.3,-13.7) node[rectangle, draw=white!0, fill=white!100]{$\textbf{.}$}
+(5.5,-13.7) node[rectangle, draw=white!0, fill=white!100]{$\textbf{.}$}
+(5.7,-13.7) node[rectangle, draw=white!0, fill=white!100]{$\textbf{.}$}
+(5.3,-16.2) node[rectangle, draw=white!0, fill=white!100]{$\textbf{.}$}
+(5.5,-16.2) node[rectangle, draw=white!0, fill=white!100]{$\textbf{.}$}
+(5.7,-16.2) node[rectangle, draw=white!0, fill=white!100]{$\textbf{.}$}

+(5,-14.5) node[rectangle, draw=white!0, fill=white!100]{$u$}

+(5.5,-17.3) node[rectangle, draw=white!0, fill=white!100]{$T_{3}$}


+(-12,-19) node{}
+(-12.8,-19.7) node{}
+(-13.55,-20.35) node{}
+(-14.25,-20.95) node{}
+(-15,-21.6) node{}
+(-12,-19) --+(-12.8,-19.7) --+(-13.55,-20.35) --+(-14.25,-20.95) --+(-15,-21.6)

+(-11.2,-19.7) node{}
+(-10.45,-20.35) node{}
+(-12,-19) --+(-11.2,-19.7) --+(-10.45,-20.35)

+(-12.75,-18) node{}
+(-11.25,-18) node{}
+(-12.75,-18) --+(-12,-19) --+(-11.25,-18)

+(-12.4,-20.1) node{}
+(-12.8,-21) node{}
+(-11.6,-20.1) node{}
+(-11.2,-21) node{}
+(-12.8,-21) --+(-12.4,-20.1) --+(-12,-19) --+(-11.6,-20.1) --+(-11.2,-21)

+(-12.2,-18.2) node[rectangle, draw=white!0, fill=white!100]{$\textbf{.}$}
+(-12,-18.2) node[rectangle, draw=white!0, fill=white!100]{$\textbf{.}$}
+(-11.8,-18.2) node[rectangle, draw=white!0, fill=white!100]{$\textbf{.}$}
+(-12.2,-20.7) node[rectangle, draw=white!0, fill=white!100]{$\textbf{.}$}
+(-12,-20.7) node[rectangle, draw=white!0, fill=white!100]{$\textbf{.}$}
+(-11.8,-20.7) node[rectangle, draw=white!0, fill=white!100]{$\textbf{.}$}

+(-12.5,-19) node[rectangle, draw=white!0, fill=white!100]{$u$}
+(-10.45,-20.75) node[rectangle, draw=white!0, fill=white!100]{$v$}

+(-12,-21.5) node[rectangle, draw=white!0, fill=white!100]{$T_{4}$}


+(-9.5,-19) node{}
+(-8.5,-19) node{}
+(-7.5,-19) node{}
+(-6.5,-19) node{}
+(-5.5,-19) node{}
+(-4.5,-19) node{}
+(-9.5,-19) --+(-8.5,-19) --+(-7.5,-19) --+(-6.5,-19) --+(-5.5,-19) --+(-4.5,-19)

+(-9,-20) node{}
+(-8,-20) node{}
+(-9,-20) --+(-8.5,-19) --+(-8,-20)
+(-8.7,-19.85) node[rectangle, draw=white!0, fill=white!100]{$\textbf{.}$}
+(-8.5,-19.85) node[rectangle, draw=white!0, fill=white!100]{$\textbf{.}$}
+(-8.3,-19.85) node[rectangle, draw=white!0, fill=white!100]{$\textbf{.}$}

+(-6,-20) node{}
+(-5,-20) node{}
+(-6,-20) --+(-5.5,-19) --+(-5,-20)
+(-5.7,-19.85) node[rectangle, draw=white!0, fill=white!100]{$\textbf{.}$}
+(-5.5,-19.85) node[rectangle, draw=white!0, fill=white!100]{$\textbf{.}$}
+(-5.3,-19.85) node[rectangle, draw=white!0, fill=white!100]{$\textbf{.}$}

+(-9.5,-18.6) node[rectangle, draw=white!0, fill=white!100]{$u$}
+(-4.5,-18.6) node[rectangle, draw=white!0, fill=white!100]{$v$}

+(-7,-20.6) node[rectangle, draw=white!0, fill=white!100]{$T_{5}$}


+(-3.5,-19) node{}
+(-2.5,-19) node{}
+(-1.5,-19) node{}
+(-0.5,-19) node{}
+(0.5,-19) node{}
+(1.5,-19) node{}
+(-3.5,-19) --+(-2.5,-19) --+(-1.5,-19) --+(-0.5,-19) --+(0.5,-19) --+(1.5,-19)

+(-3.1,-20) node{}
+(-2,-20) node{}
+(-3.1,-20) --+(-1.5,-19) --+(-2,-20)
+(-2.7,-19.95) node[rectangle, draw=white!0, fill=white!100]{$\textbf{.}$}
+(-2.5,-19.95) node[rectangle, draw=white!0, fill=white!100]{$\textbf{.}$}
+(-2.3,-19.95) node[rectangle, draw=white!0, fill=white!100]{$\textbf{.}$}

+(-1.3,-19.8) node{}
+(-1.05,-20.6) node{}
+(-0.7,-19.8) node{}
+(-0,-20.6) node{}
+(-1.05,-20.6) --+(-1.3,-19.8) --+(-1.5,-19) --+(-0.7,-19.8) --+(0,-20.6)
+(-0.7,-20.6) node[rectangle, draw=white!0, fill=white!100]{$\textbf{.}$}
+(-0.5,-20.6) node[rectangle, draw=white!0, fill=white!100]{$\textbf{.}$}
+(-0.3,-20.6) node[rectangle, draw=white!0, fill=white!100]{$\textbf{.}$}

+(0,-20) node{}
+(1,-20) node{}
+(0,-20) --+(0.5,-19) --+(1,-20)
+(0.3,-19.85) node[rectangle, draw=white!0, fill=white!100]{$\textbf{.}$}
+(0.5,-19.85) node[rectangle, draw=white!0, fill=white!100]{$\textbf{.}$}
+(0.7,-19.85) node[rectangle, draw=white!0, fill=white!100]{$\textbf{.}$}

+(-3.5,-18.6) node[rectangle, draw=white!0, fill=white!100]{$u$}
+(-1.5,-18.6) node[rectangle, draw=white!0, fill=white!100]{$v$}
+(0.5,-18.6) node[rectangle, draw=white!0, fill=white!100]{$w$}
+(1.5,-18.6) node[rectangle, draw=white!0, fill=white!100]{$x$}

+(-1.6,-20.8) node[rectangle, draw=white!0, fill=white!100]{$T_{6}$}


+(2.5,-19) node{}
+(3.5,-19) node{}
+(4.5,-19) node{}
+(5.5,-19) node{}
+(6.5,-19) node{}
+(7.5,-19) node{}
+(2.5,-19) --+(3.5,-19) --+(4.5,-19) --+(5.5,-19) --+(6.5,-19) --+(7.5,-19)

+(2.9,-20) node{}
+(4,-20) node{}
+(2.9,-20) --+(4.5,-19) --+(4,-20)
+(3.3,-19.95) node[rectangle, draw=white!0, fill=white!100]{$\textbf{.}$}
+(3.5,-19.95) node[rectangle, draw=white!0, fill=white!100]{$\textbf{.}$}
+(3.7,-19.95) node[rectangle, draw=white!0, fill=white!100]{$\textbf{.}$}

+(4.7,-19.8) node{}
+(4.95,-20.6) node{}
+(5.3,-19.8) node{}
+(6,-20.6) node{}
+(4.95,-20.6) --+(4.7,-19.8) --+(4.5,-19) --+(5.3,-19.8) --+(6,-20.6)
+(5.3,-20.6) node[rectangle, draw=white!0, fill=white!100]{$\textbf{.}$}
+(5.5,-20.6) node[rectangle, draw=white!0, fill=white!100]{$\textbf{.}$}
+(5.7,-20.6) node[rectangle, draw=white!0, fill=white!100]{$\textbf{.}$}

+(4.5,-18.6) node[rectangle, draw=white!0, fill=white!100]{$u$}
+(2.5,-18.6) node[rectangle, draw=white!0, fill=white!100]{$v$}

+(6.7,-20.5) node[rectangle, draw=white!0, fill=white!100]{$T_{7}$}


+(-8,-22.5) node{}
+(-6.75,-22.5) node{}
+(-5.5,-22.5) node{}
+(-4.25,-22.5) node{}
+(-3,-22.5) node{}
+(-8,-22.5) --+(-6.75,-22.5) --+(-5.5,-22.5) --+(-4.25,-22.5) --+(-3,-22.5)

+(-7.1,-23.5) node{}
+(-6,-23.5) node{}
+(-7.1,-23.5) --+(-5.5,-22.5) --+(-6,-23.5)
+(-6.7,-23.5) node[rectangle, draw=white!0, fill=white!100]{$\textbf{.}$}
+(-6.5,-23.5) node[rectangle, draw=white!0, fill=white!100]{$\textbf{.}$}
+(-6.3,-23.5) node[rectangle, draw=white!0, fill=white!100]{$\textbf{.}$}

+(-4.2,-23.3) node{}
+(-5,-23.3) node{}
+(-3.4,-23.9) node{}
+(-4.5,-24.1) node{}
+(-4.5,-24.1) --+(-5,-23.3) --+(-5.5,-22.5) --+(-4.2,-23.3) --+(-3.4,-23.9)
+(-4.2,-23.9) node[rectangle, draw=white!0, fill=white!100]{$\textbf{.}$}
+(-4,-23.9) node[rectangle, draw=white!0, fill=white!100]{$\textbf{.}$}
+(-3.8,-23.9) node[rectangle, draw=white!0, fill=white!100]{$\textbf{.}$}

+(-8,-22.1) node[rectangle, draw=white!0, fill=white!100]{$v$}
+(-5.5,-22.1) node[rectangle, draw=white!0, fill=white!100]{$u$}
+(-3,-22.1) node[rectangle, draw=white!0, fill=white!100]{$w$}

+(-5.5,-24.3) node[rectangle, draw=white!0, fill=white!100]{$T_{8}$}


+(-1,-22.5) node{}
+(0.25,-22.5) node{}
+(1.5,-22.5) node{}
+(2.75,-22.5) node{}
+(4,-22.5) node{}
+(-1,-22.5) --+(0.25,-22.5) --+(1.5,-22.5) --+(2.75,-22.5) --+(4,-22.5)

+(-0.25,-23.5) node{}
+(0.75,-23.5) node{}
+(-0.25,-23.5) node{} --+(0.25,-22.5) --+(0.75,-23.5)
+(0.05,-23.5) node[rectangle, draw=white!0, fill=white!100]{$\textbf{.}$}
+(0.25,-23.5) node[rectangle, draw=white!0, fill=white!100]{$\textbf{.}$}
+(0.45,-23.5) node[rectangle, draw=white!0, fill=white!100]{$\textbf{.}$}

+(2.25,-23.5) node{}
+(3.25,-23.5) node{}
+(2.25,-23.5) node{} --+(2.75,-22.5) --+(3.25,-23.5)
+(2.55,-23.5) node[rectangle, draw=white!0, fill=white!100]{$\textbf{.}$}
+(2.75,-23.5) node[rectangle, draw=white!0, fill=white!100]{$\textbf{.}$}
+(2.95,-23.5) node[rectangle, draw=white!0, fill=white!100]{$\textbf{.}$}

+(-1,-22.1) node[rectangle, draw=white!0, fill=white!100]{$u$}
+(4,-22.1) node[rectangle, draw=white!0, fill=white!100]{$v$}

+(1.5,-24.3) node[rectangle, draw=white!0, fill=white!100]{$T_{9}$}
};
\end{tikzpicture}
\end{center}\vspace{-4mm}
\caption{Family $\mathcal{G}$ of graphs $G$ with $\gamma_{rI}(G)=|V(G)|-1$.}\label{fig2}
\end{figure}
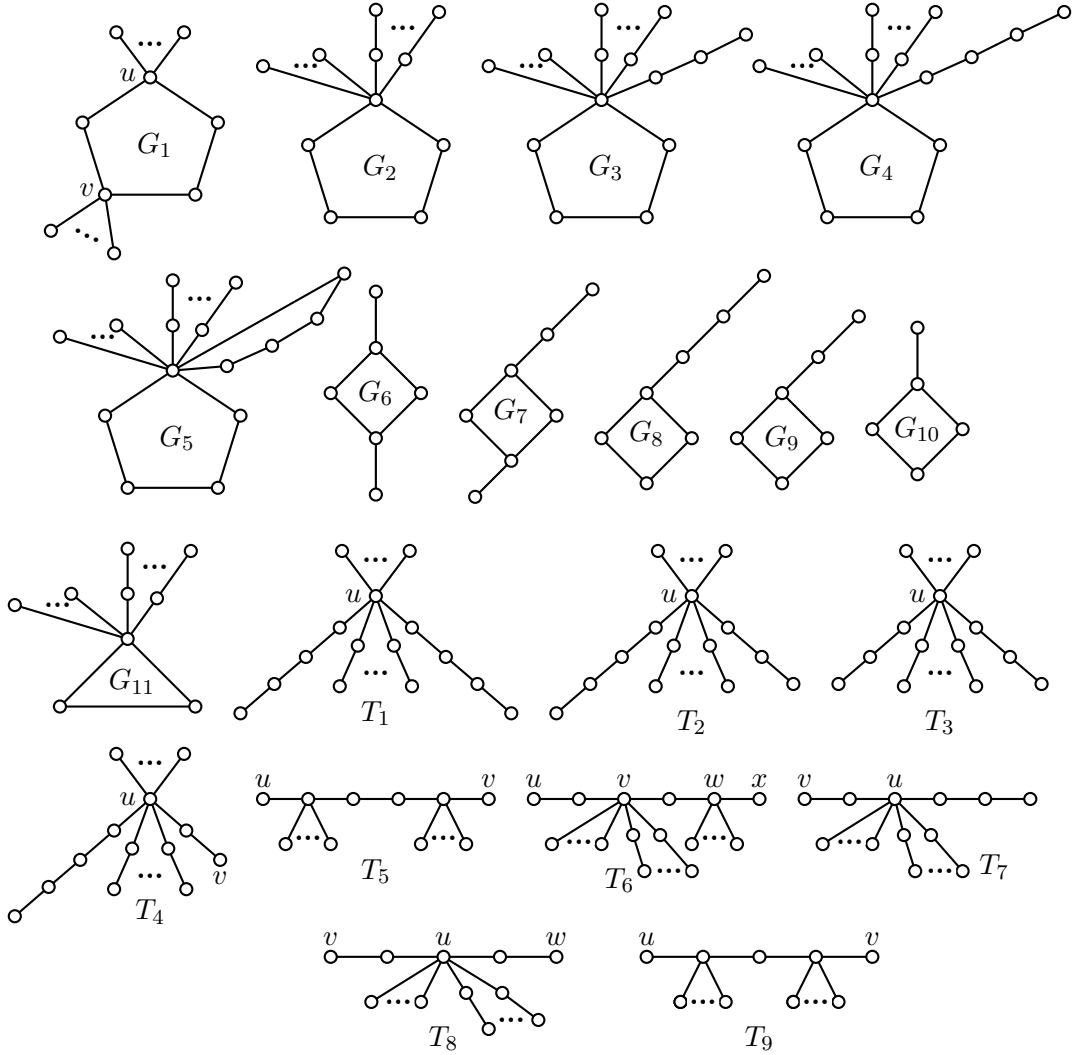

\begin{theorem}
Let $G$ be a connected graph of order $n$. Then, $\gamma_{rI}(G)=n-1$ if and only if $G\in \mathcal{G}\cup\{C_3,C_7,C_8,P_7,P_8,P_9,S_{1,q}\}$ in which $q\geq2$.
\end{theorem}
\begin{proof}
Suppose that $\gamma_{rI}(G)=n-1$. We consider two cases depending on the existence of cycles in the graph $G$.\vspace{0.75mm}

\textit{Case 1.} $G$ is not a tree. Suppose that $C_k:v_{1}v_{2}\cdots,v_{k}v_{1}$ is a cycle on $k\geq9$ vertices. Then $f(v_1)=f(v_4)=f(v_7)=2$, $f(v_2)=f(v_3)=f(v_5)=f(v_6)=f(v_8)=f(v_9)=0$, $f(v_{10})=2$ (if $k\geq 10$) and $f(x)=1$ for other vertices $x$ (if any) defines an RID function of weight $n-2$. This is a contradiction. It follows that there does not exist any cycle on $k\geq9$ vertices in $G$. Let $3\leq k\leq 8$ be the length of a longest cycle $C_{k}:v_{1}v_{2}\cdots,v_{k}v_{1}$ in $G$. Note that there is no chord $v_{i}v_{j}$ between any two vertices of $C$, for otherwise $f(v_{i})=f(v_{j})=0$ and $f(u)=1$ for any other vertex $u$ would be an RID function with weight $n-2$. This is impossible. We distinguish the following possibilities depending on the different values for $k$.\vspace{0.75mm}

\textit{Subcase 1.1.} $k\in\{7,8\}$. Suppose that there exists a vertex $x\in V(G)\setminus V(C)$ adjacent to a vertex on $C=C_8$, say $v_{1}$. Assigning $0$ to $v_1$, $v_{2}$, $v_4$ and $v_{5}$, $2$ to $v_{3}$ and $v_{6}$, and $1$ to the other vertices gives us an RID function of $G$ with weight $n-2$, which is impossible. Therefore, $G\cong C_{8}$. A similar argument shows that $G\cong C_{7}$ if $k=7$.

\textit{Subcase 1.2.} $k=6$. In such a situation, $f(v_{1})=f(v_2)=f(v_4)=f(v_5)=0$, $f(v_3)=f(v_6)=2$ and $f(x)=1$ for any other vertex $x$ defines an RID function of $G$ with weight $n-2$. Therefore, there does not exist any cycle on six vertices in $G$.\vspace{0.75mm}

\textit{Subcase 1.3.} $k=5$. Since $C=C_{5}$ has the RID number $5$, it follows that $V(G)\setminus V(C)\neq \emptyset$. If two adjacent vertices on $C$ have degree at least three, then we have $\gamma_{rI}(G)\leq n-2$ by Part $(1)$ of Lemma \ref{lem1}. Therefore for any two adjacent vertices on $C$, at least one of them has degree two. This implies that at most two (nonadjacent) vertices on $C$ have degree at least three. Let $d(x,C)=\min\{d(x,v)\mid v\in V(C)\}$. Suppose now that $d(x,C)\geq5$ for some vertex $x$ and $xabcdv_{i}$ is a path of length five for some $1\leq i\leq k$, in which $a,b,c,d\notin V(C)$. In such a situation, $(f(x),f(a),f(b),f(c),f(d),f(v_{i}))=(2,0,0,2,0,0)$ and $f(z)=1$ for the other vertices $z$ is an RID function of weigh $n-2$, which is impossible. Therefore, $d(x,C)\leq4$ for each $x\in V(G)$. We now consider two cases depending on the number $1\leq p\leq2$ of vertices of $C$ with degree at least three.\vspace{0.75mm}

\textit{Subcase 1.3.1.} $p=2$. Suppose to the contrary that there exist two vertices $x$ and $y$ with $d(x,C)=2$ and $d(y,C)=1$. Moreover, we may assume that $yv_1\in E(G)$ and $xwv_4$ is a path connecting $x$ to $C$. Then, $f(v_1)=f(v_2)=f(v_4)=f(w)=0$, $f(v_3)=f(x)=2$ and $f(z)=1$ for any other vertex $z$ defines an RID function of $G$ with weight $n-2$, a contradiction. Therefore, $d(x,C)=1$ for any vertex $x\in V(G)\setminus V(C)$. Thus, $G$ is of the form $G_{1}$ in Figure \ref{fig2}.\vspace{0.75mm}

\textit{Subcase 1.3.2.} $p=1$. Let $x$ be a vertex in $V(G)\setminus V(C)$. We claim that deg$(x)\leq2$. Suppose to the contrary that deg$(x)\geq3$, for some $x\in V(G)\setminus V(C)$. Let $w_1...w_t$ be a shortest path connecting $x$ to $C$, in which $x=w_1$ and $v_i=w_t$. If $t=2$, then $f(x)=f(v_i)=0$ and $f(z)=1$ for $z\neq x,v_i$ is an RID function with weight $n-2$, a contradiction. So, $t\geq3$. Without loss of generality, we assume that $i=1$. Suppose now that $t=3$. Then $f(x)=f(w_2)=f(v_2)=f(v_3)=0$, $f(v_1)=f(v_4)=2$ and $f(z)=1$ for the other vertices $z$ is an RID function with weight $n-2$. This is a contradiction. If $t\geq4$, then $f(x)=f(w_2)=f(v_1)=f(v_2)=0$, $f(w_3)=f(v_3)=2$ and $f(z)=1$ for any other vertex $z$ is an RID function with weight $n-2$, which is again a contradiction. Therefore, we have proved that deg$(x)\leq2$ for each $x\in V(G)\setminus V(C)$.

The discussion above guarantees that the induced subgraph $H=G[(V(G)\setminus V(C))\cup\{v_1\}]$ is isomorphic to a union of some graphs in $\{P_2,P_3,P_4,P_5,C_3,C_4,C_5\}$ such that they have only the vertex $v_1$ in common. We proceed with the following series of claims.\vspace{0.75mm}\\
\textbf{Claim A.} \textit{There is no subgraph $C_3$ in $H$.}\vspace{0.25mm}\\
\textit{Proof.} If this is not true, then there exist two adjacent vertices $x,y\in V(G)\setminus V(C)$ which are both adjacent to $v_1$. Then $f(x)=f(y)=f(v_2)=f(v_3)=0$, $f(v_1)=f(v_4)=2$ and $f(z)=1$ for the other vertices $z$ is an RID function of weight $n-2$, which is impossible. Therefore, $H$ does not have any cycle $C_3$ as a subgraph. $(\square)$\vspace{0.75mm}\\
\textbf{Claim B.} \textit{There is no subgraph $C_4$ in $H$.}\vspace{0.25mm}\\
\textit{Proof.} Suppose this is not the case. Let $v_1abcv_1$ be such a $4$-cycle. Then the assignment $f(v_1)=f(a)=f(c)=0$, $f(b)=2$ and $f(z)=1$ for any other vertex $z$ would be an RID function of $G$ of weight $n-2$, a contradiction. Thus, $H$ does not have a cycle $C_4$ as a subgraph. $(\square)$\vspace{0.75mm}\\
\textbf{Claim C.} \textit{There is at most one subgraph among $\{P_4,P_5,C_5\}$ in $H$.}\vspace{0.25mm}\\
\textit{Proof.} If there are at least two subgraphs in $H$ isomorphic to some members of $\{P_4,P_5,C_5\}$, then there are two $4$-paths $P':xw_1w_2v_1$ and $P'':yu_1u_2v_1$ in $H$. In such a situation, $f(w_1)=f(w_2)=f(u_1)=f(u_2)=f(v_2)=f(v_3)=0$, $f(x)=f(y)=f(v_1)=f(v_4)=2$ and $f(z)=1$ for any other vertex $z$ gives us an RID function of weight $n-2$, which is a contradiction. Therefore at most one path $P_4$ exists in $H$, and so, at most one graph among $\{P_4,P_5,C_5\}$ appears in $H$ as a subgraph. $(\square)$\vspace{0.75mm}

We now infer from the above argument that $G$ is one of the graphs $G_2,\cdots,G_5$ depicted in Figure \ref{fig2}.\vspace{0.75mm}

\textit{Subcase 1.4.} $k=4$. Similar to Subcase $1.3$, we have $V(G)\setminus V(C)\neq \emptyset$ in which $C=C_4$. Moreover, at most two nonadjacent vertices on $C$ have degree at least three. Let a vertex on $C$, say $v_1$, have degree at least four. This implies that $f(v_1)=f(v_2)=f(v_4)=0$, $f(v_3)=2$ and $f(x)=1$ for each $x\in V(G)\setminus V(C)$ is an RID function with weight $n-2$, which is impossible. Therefore, each vertex on $C$ has degree at most three. Let $1\leq p\leq2$ be the number on vertices on $C$ of degree three. Similar to Subcase $1.3$, all vertices in $V(G)\setminus V(C)$ have degree at most two. We now deal with two cases depending on the values for $p$.\vspace{0.75mm}

\textit{Sabcase 1.4.1.} $p=2$. Without loss of generality, we may assume that $\deg(v_1)=\deg(v_3)=3$. If there exists a path $v_1abc$ in $G$, then assigning $2$ to $v_1$ and $c$, $0$ to $a$, $b$, $v_2$ and $v_3$, and $1$ to the other vertices gives us an RID function of $G$ with weight $n-2$, which is impossible. Moreover, the existence of two paths $v_1ab$ and $v_3a'b'$ in $G$ leads to the RID function $f(b)=f(b')=2$, $f(v_1)=f(v_3)=f(a)=f(a')=0$ and $f(x)=1$ for any other vertex $x$. So, $\gamma_{rI}(G)\leq \omega(f)=n-2$. This is a contradiction. It is now easy to check that the only graphs $G$ satisfying $\gamma_{rI}(G)=|V(G)|-1$ are isomorphic to $G_6$ or $G_7$ in Figure \ref{fig2}.\vspace{0.75mm}

\textit{Sabcase 1.4.2.} $p=1$. We may assume that $\deg(v_1)=3$. Suppose that there exists a path $v_1abcd$ in $H$. It is readily seen that $f(a)=f(d)=f(v_3)=2$, $f(b)=f(c)=f(v_1)=f(v_2)=f(v_4)=0$ and $f(x)=1$ for the other vertices $x$ defines an RID function of $G$ with weight $n-2$, a contradiction. In such a situation, $G$ is isomorphic to $G_8$, $G_9$ or $G_{10}$ in Figure \ref{fig2}.\vspace{0.75mm}

\textit{Subcase 1.5.} $k=3$. If $V(G)\setminus V(C)=\emptyset$, then clearly $G=C=C_3$. Hence, we assume that $V(G)\setminus V(C)\neq \emptyset$. Again we have $\deg(x)\leq2$ for any $x\in V(G)\setminus V(C)$, by a similar fashion. On the other hand, there exists only one vertex on $C$, say $v_1$, of degree at least three by Part ($1$) of Lemma \ref{lem1}. If there is a path $xabv_1$, then the assignment $f(x)=f(v_1)=2$, $f(a)=f(b)=f(v_2)=f(v_3)=0$ and $f(z)=1$ for the other vertices $z$ would be an RID function of weight $n-2$, which is impossible. This shows that $G$ is of the form $G_{11}$ depicted in Figure \ref{fig2}.\vspace{0.75mm}

\textit{Case 2.} Suppose now that $G=T$ is a nontrivial tree. Note that Part ($2$) of Lemma \ref{lem1} implies that diam$(T)\leq8$. We distinguish the following cases depending on the possible values for diam$(T)$. In each case, we suppose that $P:v_1v_2\cdots v_r$ is a diametral path in $T$ in which $r$ $=$ diam$(T)+1$. Clearly, $v_1$ and $v_r$ are leaves.\vspace{0.75mm}

\textit{Subcase 2.1. diam$(T)=8$.} We have $\deg(v_2)=\deg(v_3)=\deg(v_4)=\deg(v_6)=\deg(v_7)=\deg(v_8)=2$ by Part ($3$) of Lemma \ref{lem1}. If $V(T)\setminus V(P)=\emptyset$, then clearly $T=P=P_9$. So, we may assume that $V(T)\setminus V(P)\neq \emptyset$. If there exists a vertex $x$ at distance three from $v_5$, then $T$ has a subtree illustrated in Part ($4$) of Lemma \ref{lem1}. This is a contradiction. This implies that $T$ is of the form $T_1$ depicted in Figure \ref{fig2}.\vspace{0.75mm}

\textit{Subcase 2.2. diam$(T)=7$.} Since $\gamma_{rI}(P_8)=7$, we may assume that $T\neq P=P_8$. Note that Part ($3$) of Lemma \ref{lem1} implies that $\deg(v_2)=\deg(v_3)=\deg(v_6)=\deg(v_7)=2$. On the other hand, Part ($1$) of the lemma and the fact that $T\neq P=P_8$ show that precisely one of $v_4$ and $v_5$, say $v_4$, has degree at least three. Moreover, there is no vertex $x\in V(T)\setminus V(P)$ at distance three from $v_4$ by Part ($4$) of Lemma \ref{lem1}. Therefore, $T\cong T_2\in \mathcal{G}$.\vspace{0.75mm}

\textit{Subcase 2.3. diam$(T)=6$.} We may assume that $T\neq P=P_7$ as $\gamma_{rI}(P_7)=6$. On the other hand, $\deg(v_2)=\deg(v_6)=2$ by Part ($3$) of Lemma \ref{lem1}. This shows that if $\deg(v_i)\geq3$, then $v_i\in \{v_3,v_4,v_5\}$. Note that none of ($a$) $\deg(v_3),\deg(v_4)\geq3$, ($b$) $\deg(v_3),\deg(v_5)\geq3$ and ($c$) $\deg(v_4),\deg(v_5)\geq3$ is the case because any of them satisfies Part ($1$) or Part ($5$) of Lemma \ref{lem1}. Therefore, precisely one of the cases $\deg(v_3)\geq3$, $\deg(v_4)\geq3$ and $\deg(v_5)\geq3$ happens. By symmetry, we may assume that at most one of $\deg(v_3)\geq3$ and $\deg(v_5)\geq3$ happens. If there exists a vertex $x\in V(T)\setminus V(P)$ at distance three from $v_3$ or $v_4$, then we derive the contradiction diam$(T)\geq7$. Moreover, every vertex $x\in V(T)\setminus V(P)$ is at distance at most two from $v_4$, for otherwise $T$ satisfies Part ($4$) of Lemma \ref{lem1}. Thus, $T\cong T_3\in \mathcal{G}$ or $T\cong T_4\in \mathcal{G}$.\vspace{0.75mm}

\textit{Subcase 2.4. diam$(T)=5$.} Since $\gamma_{rI}(P_6)=6$, it follows that $V(T)\setminus V(P)\neq \emptyset$. In such a situation,\\
($a$) no pair of adjacent vertices in $\{v_2,v_3,v_4,v_5\}$ have degree at least three simultaneously by the first part of Lemma \ref{lem1},\\
($b$) there is no vertex in $V(T)\setminus V(P)$ at distance two from $v_2$ or $v_5$ since diam$(T)=5$, and\\
($c$) there is no vertex in $V(T)\setminus V(P)$ at distance three from $v_3$ or $v_4$ since diam$(T)=5$.\\
Consequently, $T$ is of the form $T_5$, $T_6$ or $T_7$ depicted in Figure \ref{fig2}.\vspace{0.75mm}

\textit{Subcase 2.5. diam$(T)=4$.} Since $\gamma_{rI}(P_5)=5$, we have $V(T)\setminus V(P)\neq \emptyset$. In such a situation,\\
($a$) no pair of adjacent vertices in $\{v_2,v_3,v_4\}$ have degree at least three simultaneously by the first part of Lemma \ref{lem1},\\
($b$) there is no vertex in $V(T)\setminus V(P)$ at distance two from $v_2$ or $v_4$ since diam$(T)=4$, and\\
($c$) there is no vertex in $V(T)\setminus V(P)$ at distance three from $v_3$ since diam$(T)=4$.\\
Therefore, $T$ is of the form $T_8$ or $T_9$ depicted in Figure \ref{fig2}.\vspace{0.75mm}

\textit{Subcase 2.6. diam$(T)=3$.} It is easy to see that $S_{1,q}$ for $q\geq2$ is the only tree $T$ with diameter three satisfying $\gamma_{rI}(T)=|V(T)|-1$.

Conversely, it is not difficult to check that $\gamma_{rI}(G)=|V(G)|-1$ for each $G\in \mathcal{G}\cup\{C_3,C_7,C_8,P_7,\\ P_8,P_9,S_{1,q}\}$. This completes the proof of the theorem.
\end{proof}

\section{Conclusions and problems}

The concept of restrained Italian domination in graphs was initially investigated in this paper. We studied the computational complexity of this concept and proved some bounds on the RID number of graphs. In the case of trees, we characterized all trees attaining the exhibited bound. We also provided the characterizations of graphs with small or large RID numbers. We now conclude the paper  with some problems suggested by this research.\vspace{1mm}\\
$\bullet$ For any graph $G$, $\gamma_{r}(G)\leq \gamma_{rI}(G)\leq2\gamma_{r}(G)$ as already noted in Proposition \ref{Prop1}. It is worthwhile to characterize all graphs $G$ with $\gamma_{r}(G)=\gamma_{rI}(G)$ or $\gamma_{rI}(G)=2\gamma_{r}(G)$.\vspace{1mm}\\
$\bullet$ It is also worthwhile proving some other nontrivial sharp bounds on $\gamma_{r}(G)$ for general graphs $G$ or some well-known families such as bipartite, chordal, planar, triangle-free, or claw-free graphs.\vspace{1mm}\\
$\bullet$ The decision problem RESTRAINED ITALIAN DOMINATION is NP-complete even for bipartite graphs, chordal graphs and planar graphs with maximum degree five, as proved in Theorem \ref{Comp}. By the way, there might be some polynomial-time algorithms for computing the RID number of some well-known families of graphs, for instance, trees. Is it possible to construct a polynomial-time algorithm so as to compute $\gamma_{rI}(T)$ for any tree $T$? 


\end{document}